\documentclass{article}
\usepackage[utf8]{inputenc}
\usepackage[a4paper, total={6in, 8in}]{geometry}
\usepackage{authblk}
\usepackage{amsthm}
\usepackage{amsmath}
\usepackage{amssymb}
\usepackage{graphicx}
\usepackage{bbm}
\usepackage{braket,amsfonts,amssymb}
\usepackage[linesnumbered,ruled]{algorithm2e}
\usepackage{comment}
\usepackage{amsopn}
\usepackage{bm}
\usepackage{mathrsfs}  
\usepackage[mathscr]{euscript}
\usepackage{bm,nicefrac}
\usepackage{extarrows}
\usepackage{caption}
\usepackage{subcaption}
\usepackage{xcolor}
\usepackage[linktoc=all]{hyperref}
\usepackage{titlesec}
\hypersetup{
    colorlinks,
    citecolor=blue,
    filecolor=black,
    linkcolor=blue,
    urlcolor=black
}

\usepackage{enumerate}
\usepackage{enumitem}

\usepackage[normalem]{ulem}

\def\bo{{\bf 1}}

\def\d{{\bf{d}}}

\def\g{g}
\def\h{h}

\def\n{n}
\def\p{p}

\def\s{s}
\def\sigmainn{\rho_{\text{rp}}}
\def\sigmaout{\rho_{\text{link}}}
\def\t{t}
\def\u{{\bf u}}
\def\v{{\bf v}}
\def\w{{\bf w}}
\def\x{{\bf x}}

\def\y{{\bf y}}
\def\z{{\bf z}}
\def\zprod{\widetilde{\z}}

\def\we{w}
\def\winfty{\w_{\infty}}
\def\wT{\w_{T}}
\def\wprodT{\wprod_{T}}
\def\wprod{\widetilde{\w}}
\def\wprode{\widetilde{\we}}
\def\wprodinfty{\wprod_\infty}

\def\wDomainOneD{{\mathcal{D}}}
\def\wDomainRP{{\mathcal{D}_{\text{rp}}^N}}
\def\wDomainRPlim{{\mathcal{D}_{\text{rp}}^\text{lim}}}
\def\wDomainRPOneD{{\mathcal{D}_{\text{rp}}}}

\def\ye{y}
\def\ze{z}

\def\A{{\bf A}}
\def\Ae{A}
\def\B{{\bf B}}

\def\IRERM{\mathcal{L}}
\def\F{F}

\def\H{H}

\def\J{J}

\def\L{L}

\def\LossM{\widetilde{\mathcal{L}}}
\def\M{M}
\def\N{N}

\def\R{\mathbb{R}}

\def\T{T}

\def\0{\boldsymbol{0}}
\def\1{\boldsymbol{1}}

\DeclareMathOperator*{\argmin}{arg\,min}

\renewcommand{\L}{L}

\renewcommand{\epsilon}{\varepsilon}

\newcommand{\wDomainCirc}{\mathcal{D}^N_\circ}
\newcommand{\wDomainCircOneD}{\mathcal{D}_\circ}

\newcommand{\artanh}{\textnormal{artanh}}
\newcommand{\sign}{\mathrm{sign}}

\newcommand{\revFinal}[1]{#1}

\providecommand{\keywords}[1]{\textbf{\textit{Keywords --- }} #1}

\newtheorem{theorem}{Theorem}[section]

\newtheorem{definition}[theorem]{Definition}
\newtheorem{lemma}[theorem]{Lemma}
\newtheorem{corollary}[theorem]{Corollary}
\newtheorem{assumption}[theorem]{Assumption}
\theoremstyle{remark}
\newtheorem{remark}[theorem]{Remark}

\title{How to induce regularization in linear models: \\ A guide to reparametrizing gradient flow}
\author[1,3]{Hung-Hsu Chou}
\author[2,3]{Johannes Maly\thanks{Corresponding author, email: \url{maly@math.lmu.de}}}
\author[4]{Dominik St\"oger}
\affil[1]{\normalsize School of Computation, Information and Technology,
TU Munich, Germany}
\affil[2]{\normalsize Department of Mathematics, LMU Munich, Germany}
\affil[3]{Munich Center for Machine Learning (MCML)}
\affil[4]{Mathematical Institute for Machine Learning and Data Science (MIDS), KU Eichst\"att-Ingolstadt}
\date{}

\begin{document}

\maketitle

\begin{abstract}
    In this work, we analyze the relation between reparametrizations of gradient flow and the induced implicit bias in linear models, which encompass various basic regression tasks.
    \revFinal{In particular, we study how reparametrization, loss function, and link function influence the convergence behavior and implicit bias of gradient flow.}
    Our results provide conditions under which the implicit bias can be well-described and convergence of the flow is guaranteed. 
    We furthermore show how to use these insights for designing reparametrization functions that lead to specific implicit biases which are closely connected to $\ell_p$- or trigonometric regularizers.
\end{abstract}


\section{Introduction}
\label{sec:Introduction}

Modern machine models are often highly overparameterized, i.e., the number of parameters is much larger than the number of samples and typically trained via first-order methods such as (stochastic) gradient descent.
Due to overparameterization, multiple local minimia of the loss function exist and many of them might generalize poorly to new data.
However, (stochastic) gradient descent tends to converge towards parameter configurations that generalize well, even when no regularization is applied at all \cite{zhang17}.
This tendency of (stochastic) gradient descent to prefer certain minimizers over others is commonly called implicit bias/regularization and appears to be a key for developing a theoretical understanding of modern machine learning models such as deep neural networks. 
However, a rigorous mathematical analysis of the implicit bias in general neural network models seems not within the reach of current techniques.
Thus, in the past few years, many works have focused on simplified settings like overparameterized linear models and matrix factorization/sensing, cf.\ Section \ref{sec:related_work}. 
Interestingly, many of the puzzling empirical phenomena observed while training neural networks are observed in these simpler models as well.
However, these models have the advantage that they are more amenable to theoretical analysis.
What is more, all of these works demonstrate an implicit bias of vanilla gradient flow/descent towards solutions of simple and well-understood structures \revFinal{such as} sparsity or low rank.

Let us make this precise in the exemplary case of sparse recovery. The problem is about recovering an unknown vector $\w_\star \in \R^\N$ from few linear observations
\begin{align}
\label{eq:CS}
    \y = \mathbf A\w_\star \in \R^\M,
\end{align}
where $\mathbf A \in \R^{\M \times \N}$. A basic approach to solve \eqref{eq:CS} via gradient descent is to compute solutions to
\begin{align}
\label{eq:LS}
    \min_{\z \in \R^\N} \| \A\z - \y \|_2^2.
\end{align}
For sufficiently small step-sizes, the gradient descent trajectory $\w(\t)$ initialized at $\w_0 = \boldsymbol{0}$ converges to the least-square fit $\w_\infty := \lim_{\t\to\infty} \w(\t) = \A^\dagger\y$, which minimizes the $\ell_2$-norm among all solutions of \eqref{eq:LS}. Here the matrix $\A^\dagger \in \R^{\M\times \N}$ denotes the Moore-Penrose pseudoinverse of $\A$.
On the one hand, we know that in general $\w_\infty \neq \w_\star$ if $\M \ll \N$. On the other hand, if $\w_\star$ is $s$-sparse\footnote{A vector $\z \in \R^\N$ is called $s$-sparse if at most $s$ of its entries are non-zero.} and $\mathbf A$ behaves sufficiently well with $\M \gtrsim s \log(\N/s)$, signal processing theory \cite{candes2006robust,donoho2006compressed,foucart2013compressed} shows that $\w_\star$ can be uniquely identified from \eqref{eq:CS} by solving the constrained $\ell_1$-minimization problem
\begin{align}
\label{eq:LASSO}
    \min_{\z \in \R^\N} \| \z \|_1 \qquad \text{subject to} \quad \A\z = \y,
\end{align}
where the $\ell_1$-norm serves as an explicit regularizer for sparsity.
\revFinal{Recent works \cite{vaskevicius2019implicit,woodworth2020kernel,chou2021more,li2021implicit} proved that instead of solving \eqref{eq:LASSO}, which is convex but nonsmooth due to the $\ell_1$-norm, one can solve the overparametrized objective}
\begin{align} \label{eq:Lover}
     \min_{\z^{(1)}, \dots, \z^{(L)} \in \R^\N}
    \Big\| \mathbf A \big( \z^{(1)} \odot \cdots \odot \z^{(L)} \big) -\y   \Big\|_2^2,
\end{align}
\revFinal{which is smooth in the parameters but nonconvex,}
via vanilla gradient flow/descent. Here, $\odot$~denotes the (entry-wise) Hadamard product. Indeed, several works \cite{chou2021more,li2021implicit,woodworth2020kernel} guarantee that if all Hadamard factors are initialized \revFinal{identically and} close to the origin, then gradient flow/descent $(\w^{(1)}(\t), \dots, \w^{(L)}(\t))$ converges to $(\w^{(1)}_\infty, \dots, \w^{(L)}_\infty)$ with $\wprodinfty := \w^{(1)}_\infty \odot \cdots \odot \w^{(L)}_\infty$ being an approximate $\ell_1$-minimizer among all possible solutions of \eqref{eq:CS} 
\revFinal{in the same orthant as the initialization. The orthant restriction is due to saddle points at zeros, and \revFinal{can easily be overcome by extending the parameter space, i.e., using the objective} $\| \mathbf A ( \w_+^{(1)} \odot \cdots \odot \w_+^{(L)}- \w_-^{(1)} \odot \cdots \odot \w_-^{(L)}) -\y \|$ as in \cite{chou2021more}. In this case, the difference of the products $\wprod_\infty^{\pm}:=\w^{(1)}_{\infty,+} \odot \cdots \odot \w^{(L)}_{\infty,+}-\w^{(1)}_{\infty,-} \odot \cdots \odot \w^{(L)}_{\infty,-}$ approximates $\w_\star$.} Such phenomena led to the conclusion that the (artificially) introduced overparametrization allows gradient descent to implicitly identify the underlying structure of the problem, i.e., sparsity of $\w_\star$. Interestingly, the formulation in \eqref{eq:Lover} is equivalent to the training of deep linear networks with diagonal weight matrices \cite{woodworth2020kernel} by viewing $\z^{(\ell)}$ as the diagonal of the $\ell$-th layer weight matrix.

It is important to note that several works \cite{woodworth2020kernel,Gissin2019Implicit,chou2021more} make use of a simple observation in their analysis: if all $\z^{(k)}$ are initialized with the same vector $\w_0$ and gradient flow is applied to \eqref{eq:Lover}, then the iterates $\w^{(k)}$ stay identical over time and \eqref{eq:Lover} is equivalent to solving
\begin{align}
\label{eq:Lover_simplified}
    \min_{\z \in \R^\N}
    \| \mathbf A \z^{\odot L} -\y \|_2^2
\end{align}
via gradient flow with initialization $\w_0$, where $\z^{\odot L} = \z \odot \cdots \odot \z$ denotes the $L$-th Hadamard power of $\z$, see e.g.\ \cite[Lemma 12]{chou2021more}. In other words, the implicit bias of overparametrized gradient flow in \eqref{eq:Lover} is linked to the implicit bias of \emph{reparametrized gradient flow} in \eqref{eq:Lover_simplified}. The only difference between \eqref{eq:LS} and \eqref{eq:Lover_simplified} lies in a reparametrization of the input space via the univariate non-linear map $z \mapsto z^L$ that is applied entry-wise. A related reduction technique appears in the analysis of deep \emph{linear} networks under the name \emph{balanced initialization} \cite{hardt2016identity,arora2018optimization}. Therefore, to analyze the implicit bias of gradient descent in general neural network training, it seems necessary to fully understand the relation between over- and reparametrization of the input space in simplified settings first.

\subsection{Contribution}\label{sec:contribution}

In this work, we consider the setting of linear models as outlined above and focus on two central questions. \emph{Which implicit regularization other than $\ell_1$-norm minimization can be induced via reparametrized gradient flow/descent? \revFinal{And how} does the choice of the loss function influence gradient flow?}

\revFinal{We study a class of reparametrized linear prediction problems inspired by generalized linear models \revFinal{(GLMs)} \cite{dobson2018introduction,hardin2007generalized,mccullagh2019generalized,hastie2017generalized}. Our setting overlaps with standard GLM-type regression models when $L$ and $\sigmaout$ are chosen accordingly, but we allow more general losses.}
We only consider link functions that act entry-wise on vectors. Under this restriction the training of such models corresponds to the training of a single neuron.

\begin{definition}
\label{definition:IRERM}
    Let $\N,\M\in\mathbb{N}$, $\A\in\R^{\M\times\N}$ and $\y\in\R^\M$. Let $\sigmaout \colon \R \to \R$ be a \emph{link function} acting entry-wise on vectors and let $L \colon \R^M \times \R^\M \to \R$ be a general loss function. %
    We consider the minimization problem
    \begin{align}\label{eq:GeneralLinearModel}
        \min_{\widetilde \z \in \R^\N} L(\sigmaout(\A\widetilde \z),\y).
    \end{align}
    For any \emph{reparametrization function} $\sigmainn\colon\R\to\R$ acting entry-wise on vectors, a reparametrized gradient flow $\w \colon [0,T) \to \R^\N$ on \eqref{eq:GeneralLinearModel} is any solution to the dynamics
    \begin{align}\label{eq:gd_IRERM}
       \w'(\t) = -\nabla\mathcal{L}(\w(\t))
        ,\quad \w(0)=\w_0,
    \end{align}
    where $\w_0 \in \R^\N$ and $\IRERM \colon \R^\N \to \R$ is defined as\footnote{\revFinal{Whenever we write $\nabla \IRERM$, we assume that the composition defining $\IRERM$ is differentiable; the precise regularity assumptions on $L$, $\sigmaout$, and $\sigmainn$ are stated in Section \ref{sec:MainResults}.}}
    \begin{align}\label{eq:IRERM}
        \IRERM(\z):= L\Big( \sigmaout \big( \A\sigmainn(\z) \big), \y \Big).
    \end{align}
    We say that the reparametrized flow $\w: (0,T) \rightarrow \mathbb{R}^\N$ is \emph{regular} if we have that\\ $$|\{ t \ge 0 \colon \sigmainn'(\w(t)_i) = 0 \text{ for some } i \in [\N]\}| < \infty,$$
    \revFinal{where by $\vert \Omega \vert$ we denote the cardinality of a
    set $\Omega$. We furthermore call a gradient flow trajectory $\w : [0,T) \to \R^\N$ \emph{maximal} if, for any $\varepsilon > 0$ and any solution $\bar\w : [0,T+\varepsilon) \to \R^\N$ of \eqref{eq:gd_IRERM}, we have $\bar\w|_{[0,T)} \neq \w$. }
\end{definition}

\begin{remark}
\label{rem:MainDefinition}
    Let us briefly comment on Definition \ref{definition:IRERM}:
    \begin{enumerate}
        \item[(i)] Note that \eqref{eq:IRERM} reduces to \eqref{eq:Lover_simplified} if we set $\sigmaout(z) = z$, $\sigmainn(z) = z^L$, and $L(\z,\z') = \| \z - \z' \|_2^2$. 
        \item[(ii)] 
        \revFinal{
        \textbf{Existence of gradient flow solutions:}
        In the following, we will make only mild assumptions on the regularity of the functions $\sigmainn$, $\sigmaout$, and $L$, see Section \ref{sec:MainResults}.
        More precisely, we will only require these functions to be continuously differentiable.
        Since it follows from equations \eqref{eq:gd_IRERM} and \eqref{eq:IRERM} that
        \begin{equation}\label{equ:intern33}
            \w'(\t)
        = - \Big[ \A^T \nabla_\v L\big( \sigmaout(\v), \y \big) \big|_{\v = \A\sigmainn(\w(t))} \Big] \odot \sigmainn'(\w(t)),
        \end{equation}
        the existence of at least one solution to the ODE is guaranteed by Peano's theorem as the right-hand side of \eqref{equ:intern33} is continuous in $\w$.
        These mild assumptions allow us to analyze a wider class of reparametrizations,
        which is the main purpose of the paper at hand.}
        
        \item[(iii)] 
        
        \textbf{\revFinal{Uniqueness and regular flows}:}
        \revFinal{
        Our regularity assumptions will not necessarily suffice to guarantee uniqueness of the solution on \eqref{eq:gd_IRERM}.
        Indeed, since the right-hand side of equation \eqref{equ:intern33} is not Lipschitz continuous in general and the classical Picard-Lindel\"of theorem is not applicable.
        More importantly, in our setting, non-uniqueness of the trajectory arises precisely 
        when entries of the gradient flow 
        hit points $w$ with $\sigmainn'(w)=0$ since at such points the vector field in equation \eqref{equ:intern33} vanishes in the corresponding coordinates and there is possibility that the flow can bifurcate. Trajectories that reach these points may “pause” for an indefinite time interval, 
        leading to infinitely many different solutions.\\
       Different solutions to the gradient flow ODE lead to different implicit biases. 
       In our paper, we focus on \textit{regular solutions},
       i.e., solutions where the
       entries of the gradient flow solution do not stop
       for a period of time.
       A key contribution in our paper is to characterize 
       the implicit bias of \textit{regular flows}.\\
       The motivation for considering regular flows is that  we observe in our experiments that gradient descent behaves 
       similar to regular solutions for small step size,
       which allows us to predict the implicit bias 
       of gradient descent with sufficiently small step size.
       It would be an interesting future research direction
        to study whether one can classify the implicit bias for gradient flow solutions where single neurons/weights "pause", i.e., stop moving for some period time. 
        }


        \item[(iv)]

        \revFinal{
        \textbf{Discretization of gradient flow:}
        In cases where there is a unique solution of a gradient flow equation,
        the gradient descent solution converges towards the flow solution when the step size converges towards zero.
        However, in this paper we are interested in scenarios
        where the underlying gradient flow has not
        necessarily a unique solution.
        As our experiments below indicate,
        gradient descent with vanishing step sizes converges to a regular solution. 
        Indeed, we observe no "pausing"
        and we observe that gradient descent converges 
        towards the limit points predicted by our theory.
        Intuitively, this makes sense
        since the matrices in our experiment are chosen at random
        and the probability of hitting a point $w$ with
        $\sigmainn' (w)=0$ is expected to be zero.
        }
    \end{enumerate}
\end{remark}

\noindent \textbf{Our contribution:} 
Our results show that the implicit bias of gradient flow in \eqref{eq:gd_IRERM} is governed by the choice of the reparametrization $\sigmainn$ and the initialization $\w_0$, whereas the convergence behavior depends in addition on the regularity of the link function $\sigmaout$ and the loss $L$.
Our contribution can be summarized as follows:
\begin{itemize}
    \item Theorem \ref{theorem:vector_IRERM}: For a wide class of reparametrizations $\sigmainn$, we explicitly characterize the implicit bias of $\w$ in terms of a Bregman divergence $D_F$ with a potential function $F$ which depends on $\sigmainn$.
    In contrast to previous work \cite{gunasekar2018char,amid2020reparameterizing,wu2021implicit}, which already described a connection between the implicit bias of gradient flow and particular Bregman divergences, our assumptions on $\sigmainn$ allow us to easily calculate the Bregman divergence $D_F$ from the reparametrization $\sigmainn$. 
    An analogous relation appeared recently in \cite{Li2022implicit}.
    Due to the more general setting considered therein, the results however require more restrictive assumptions.
    \revFinal{\revFinal{In particular, Li et al.\ \cite{Li2022implicit} assume regular parametrizations, meaning that the Jacobian is everywhere of full rank, whereas our framework permits degenerate} reparametrizations whose derivative may vanish at isolated points.}
    \item Theorem \ref{thm:Convergence}: By setting $\sigmaout = \mathrm{Id}$, we provide conditions under which a solution $\w \colon [0,T) \to \R^\N$ can be extended and convergence of $\w$ to a global minimizer of the loss can be guaranteed.  We furthermore describe the convergence rates in dependence on the properties of $L$. 
    Many existing works on implicit regularization explicitly assume $T = \infty$ and convergence of $\w$, cf.\ Section \ref{sec:related_work}. 
    Note that parts of the proof techniques we use for the convergence rate analysis are conceptually close to the ones used in \cite{tzen2023variational} which appeared recently and analyzes mirror flow.
    However, the authors of \cite{tzen2023variational} consider a simpler setting in which the objective function is strictly convex  and, in particular, the global minimizer is unique. 
    \revFinal{Moreover, the sufficient condition for boundedness of the gradient flow trajectory (Theorem \ref{thm:Convergence} \ref{thm:Convergence_v}) has not appeared in previous work to the best of our knowledge.}
\end{itemize}

With this work we aim to unify various existing partial results on implicit bias of gradient flow on regression models (resp.\ diagonal linear networks) into a coherent theme and make them easily accessible. Especially the latter goal motivates the title of our paper. Apart from deepening the understanding of the influence of over-/reparametrization on gradient flow, we provide a practicable
approach to design reparametrizations in regression models that implicitly encode specific regularizers into (vanilla) gradient flow/descent. Before diving into the technical details, we would like to showcase two special instances of such implicit biases, which can easily 
\revFinal{be derived} from our results. They shall serve as a blueprint for further adaptions. To keep the presentation concise in Section \ref{sec:Introduction}, we restrict ourselves to $\sigmaout = \mathrm{Id}$, only consider vanishing initialization, and do not discuss convergence rates.  
Those details can be found in Section \ref{sec:MainResults}.
Let us summarize the assumptions that we use in both showcase theorems.

\begin{assumption}
\label{assumption:Intro}
    For the
    remainder of Section \ref{sec:contribution}, we assume the following simplified setting:
    \begin{itemize}
        \item $\sigmaout = \mathrm{Id}$.
        \item $L$ is continuously differentiable, convex in the first argument, and satisfies $\min L = 0$ with $L(\z,\z') = \min_{\z,\z' \in \R^\M} L(\z,\z')$ if and only if $\z = \z'$.
    \end{itemize}
\end{assumption}

Our first showcase result illustrates how to induce $\ell_p$-regularization for $p \in (1,2)$.

\begin{theorem}
\label{thm:PolynomialRegularizationIntroduction}
    Under Assumption \ref{assumption:Intro}, 
    \revFinal{assume} $\sigmainn(z) = \sign(z) |z|^{\frac{2}{p}}$, for $p \in (1,2)$, 
    and let $\w \colon [0,T) \to \R^\N$ be any %
    regular
    solution to \eqref{eq:gd_IRERM} with 
    $\w_0 = \sigmainn^{-1} (\alpha \bo)$, for $\alpha > 0$, $\A\in\mathbb{R}^{\M\times\N}$, and $\y\in \mathrm{im}(\A)$,
    i.e., the column span of the matrix $\A$.

    Then, $\w$ is not maximal. 
    Furthermore, if $\w$ can be extended to a maximal 
    regular
    solution $\w: [0,\infty) \rightarrow \R$,\footnote{In the following we will slightly abuse notation and not relabel solutions of \eqref{eq:gd_IRERM} when we consider extensions to larger intervals.} then
    $\w_{\infty,\alpha} := \lim_{\t \to \infty} \w(\t) \in \R^\N$ exists and is a global minimizer of $\mathcal L$.    
    Set $\wprod_{\infty,\alpha} := \sigmainn(\w_{\infty,\alpha})$ and
    \begin{equation*}
        \mu:= \min_{\substack{\widetilde\z\in\R^\N,\\ \revFinal{L(\A\widetilde\z,\y)}=0}} \| \widetilde\z \|_p,
    \end{equation*}
    and assume that 
    \begin{equation*}
        \alpha
        \le \frac{1}{2^{1/(p-1)} N^{1/p}} \mu.
    \end{equation*}
    Then it holds that
    \begin{equation}
    \label{eq:PolynomialRegularizationIntro}
        \| \wprod_{\infty,\alpha} \|_p
        \le
        \left(1+ 4 N^{1-1/p} \left( \frac{ \alpha }{ \mu }\right)^{p-1}    \right) \cdot \mu.
    \end{equation}
    In other words, for small $\alpha$, the reparametrized flow $\sigmainn(\w(\t))$ approximately minimizes the $\ell_p$-norm among all (reparametrized) global minimizers $\widetilde \z = \sigmainn(\z)$ of $\revFinal{L(\A \widetilde\z,\y)}$.

\end{theorem}

\begin{figure}
     \centering
     \begin{subfigure}[b]{0.45\textwidth}
         \centering
         \includegraphics[width=\textwidth]{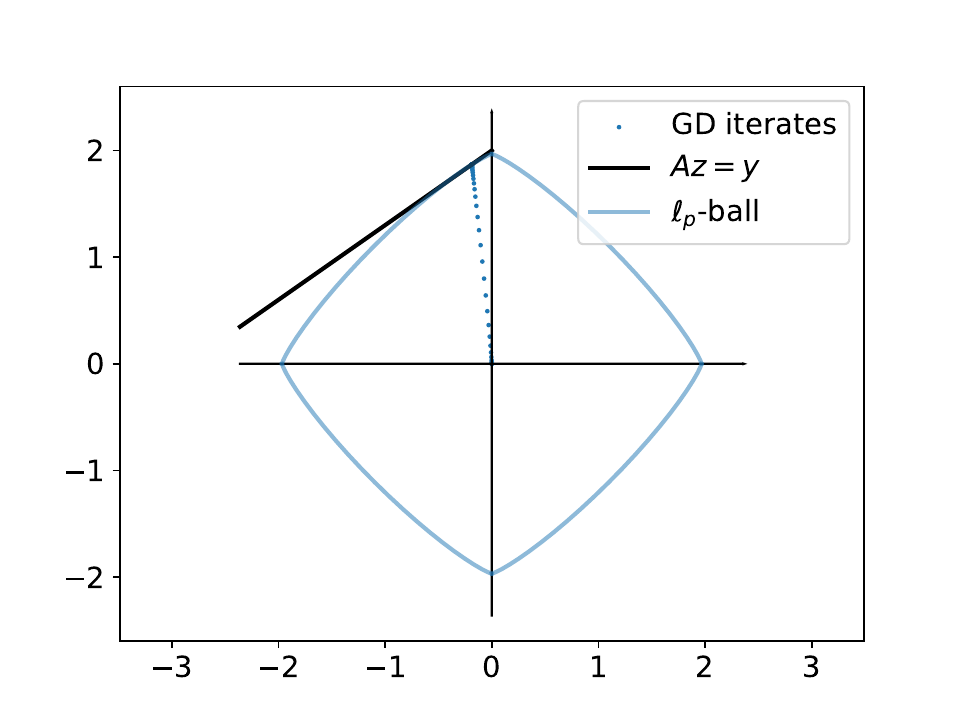}
         \caption{$p = 1.2$}
         \label{fig:1_p=1.2_A1}
     \end{subfigure}
     \hfill
     \begin{subfigure}[b]{0.45\textwidth}
         \centering
         \includegraphics[width=\textwidth]{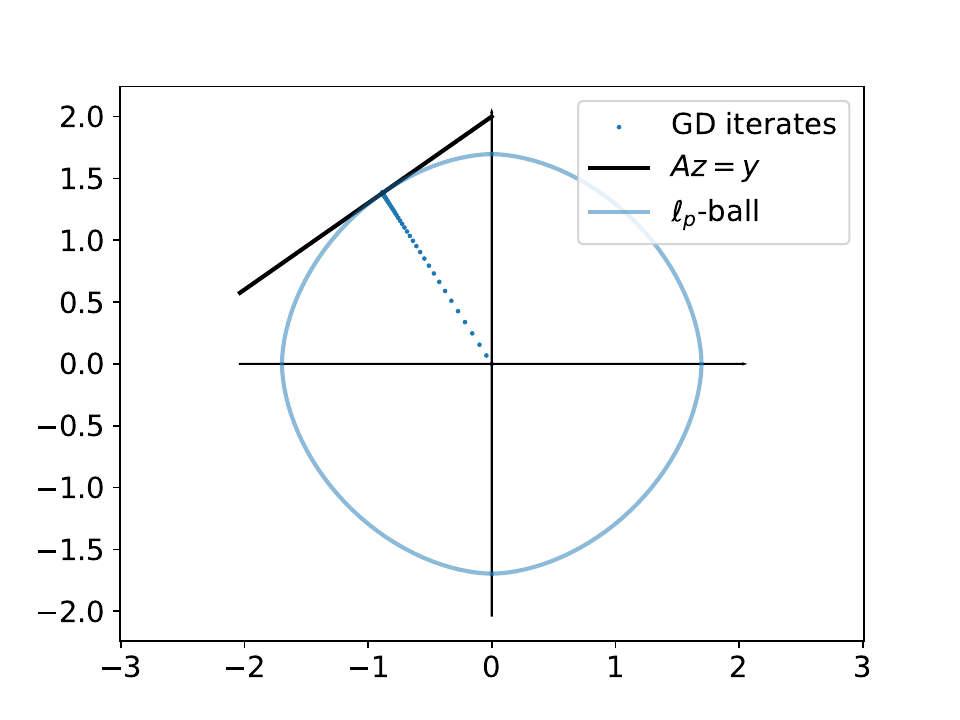}
         \caption{$p = 1.8$}
         \label{fig:1_p=1.8_A1}
     \end{subfigure} \\
     \begin{subfigure}[b]{0.45\textwidth}
         \centering
         \includegraphics[width=\textwidth]{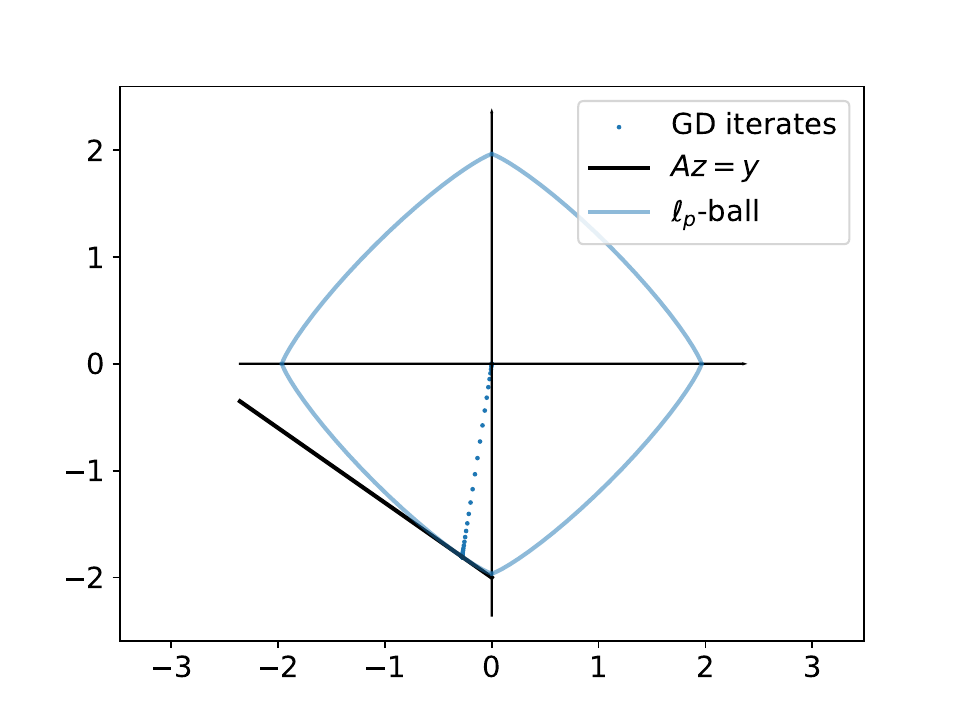}
         \caption{$p = 1.2$}
         \label{fig:1_p=1.2_A2}
     \end{subfigure}
     \hfill
     \begin{subfigure}[b]{0.45\textwidth}
         \centering
         \includegraphics[width=\textwidth]{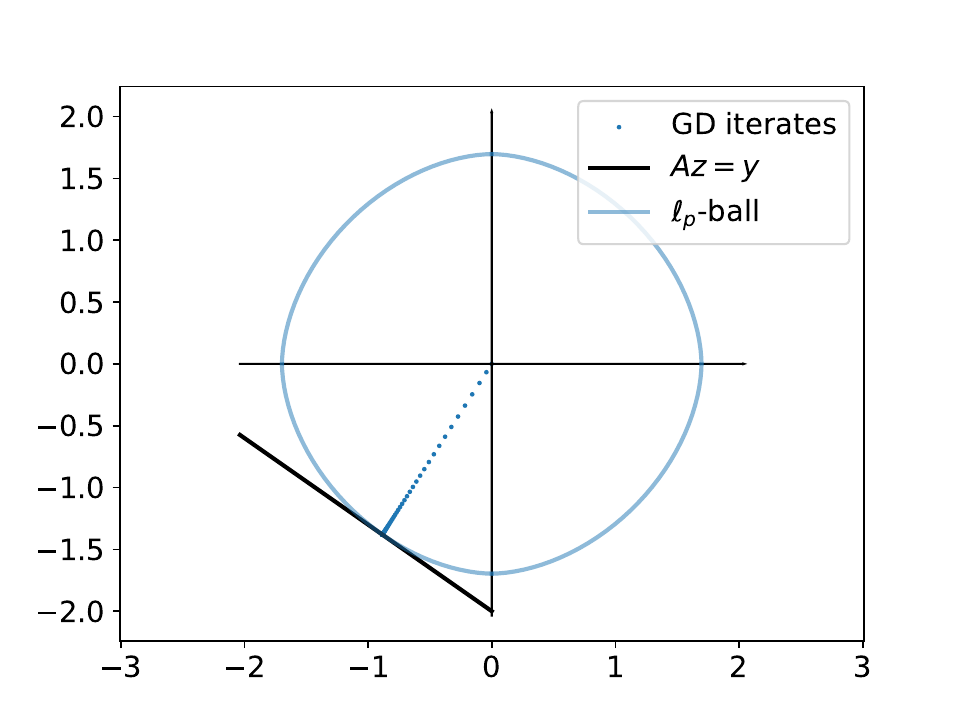}
         \caption{$p = 1.8$}
         \label{fig:1_p=1.8_A2}
     \end{subfigure}
    \caption{\revFinal{Implicit regularization of gradient descent for $\sigmainn(z) = \sign(z)|z|^{\frac{2}{p}}$ with learning rate $10^{-4}$ initialized at $(10^{-4},10^{-4})^T$. We set $\y = 2$, and $L(z,y) = |z-y|^{2}$. In Figures \ref{fig:1_p=1.2_A1} and \ref{fig:1_p=1.8_A1}, we set $\A = (-0.7,1) \in \R^{1\times 2}$. In Figures \ref{fig:1_p=1.2_A2} and \ref{fig:1_p=1.8_A2}, we set $\A = (-0.7,-1) \in \R^{1\times 2}$. The depicted $\ell_p$-ball is scaled to the final GD iterate. }
    }
    \label{fig:Polynomial}
\end{figure}

\begin{figure}
     \centering
     \begin{subfigure}[b]{0.45\textwidth}
         \centering
         \includegraphics[width=\textwidth]{Figure1_p=1,8_A1.pdf}
         \caption{$L(z,y) = |z-y |^{2}$}
         \label{fig:2_L=2}
     \end{subfigure} \\
     \begin{subfigure}[b]{0.45\textwidth}
         \centering
         \includegraphics[width=\textwidth]{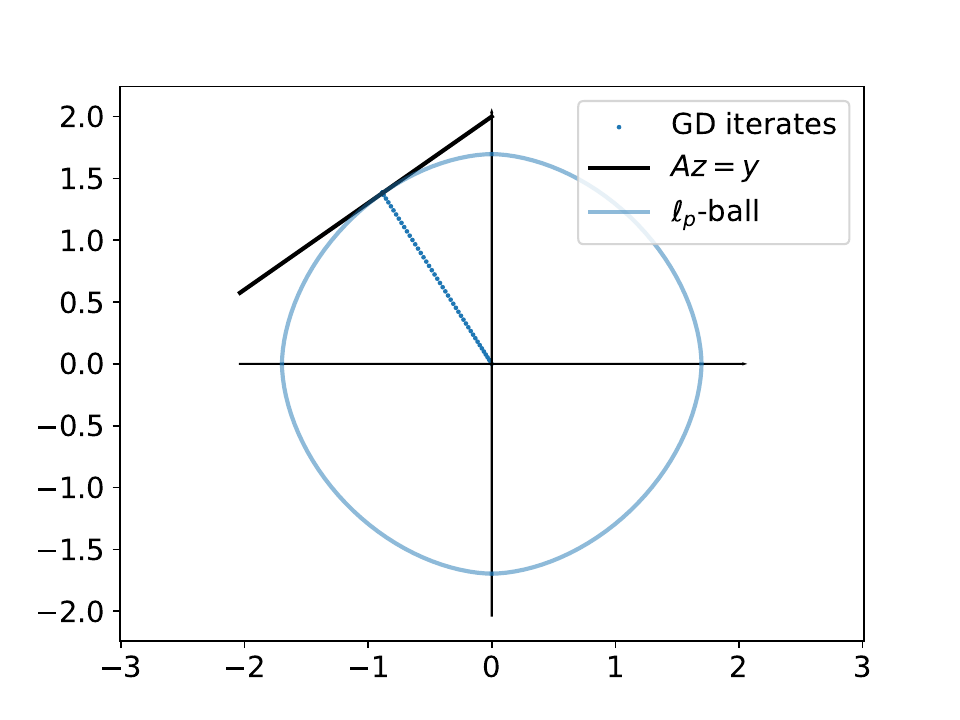}
         \caption{$L(z,y) = |z-y |^{1.1}$}
         \label{fig:2_L=1.1}
     \end{subfigure}
     \hfill
     \begin{subfigure}[b]{0.45\textwidth}
         \centering
         \includegraphics[width=\textwidth]{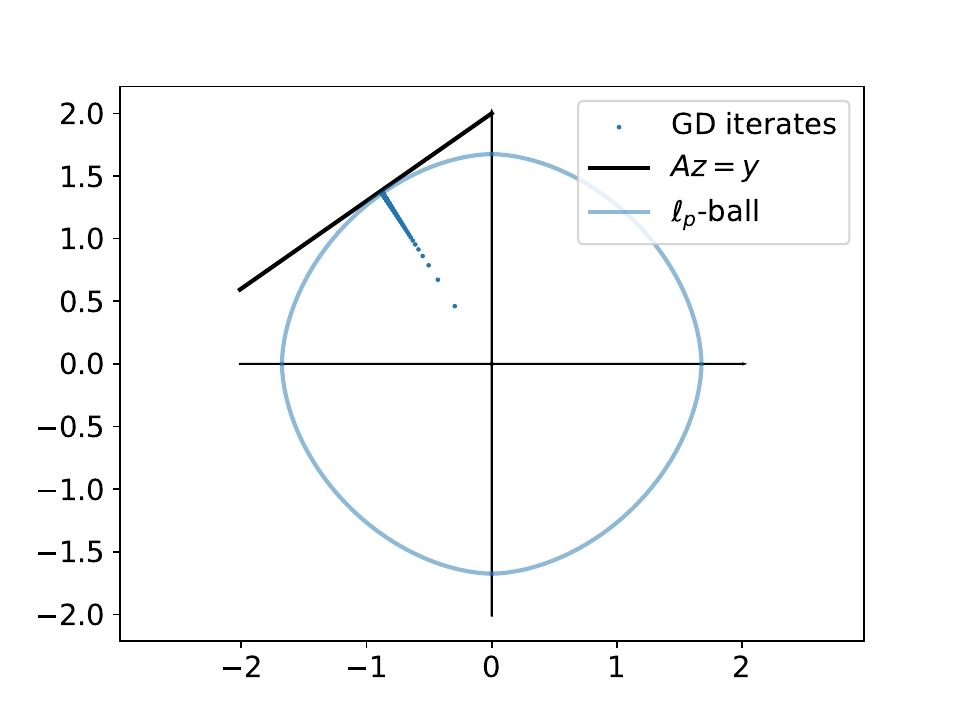}
         \caption{$L(z,y) = |z-y |^{4}$}
         \label{fig:2_L=4}
     \end{subfigure}
    \caption{\revFinal{Implicit regularization of gradient descent for $\sigmainn(z) = \sign(z)|z|^{\frac{2}{p}}$, for $p=1.8$, with learning rate $10^{-4}$ initialized at $(10^{-4},10^{-4})^T$. We set $\y = 2$ and $\A = (-0.7,1) \in \R^{1\times 2}$. For reference, Figure \ref{fig:2_L=2} equals Figure \ref{fig:1_p=1.8_A1}. The depicted $\ell_p$-ball is scaled to the final GD iterate. }
    }
    \label{fig:Polynomial_Varying_L}
\end{figure}

Theorem \ref{thm:PolynomialRegularizationIntroduction} is a direct consequence of Corollaries \ref{cor:PolynomialRegularization} and \ref{cor:PolynomialRegularizationConvergenceRate} below.
It interpolates between established results for implicit $\ell_1$-minimization \cite{vaskevicius2019implicit,woodworth2020kernel,chou2021more} and the implicit $\ell_2$-minimization in \eqref{eq:LS} which corresponds to the case $p=2$. Note that Theorem \ref{thm:PolynomialRegularizationIntroduction} \revFinal{complements} the reparametrizations in \eqref{eq:Lover_simplified} by setting $p = \frac{2}{L}$, for $L < 2$, and recovers the implicit bias characterization in \cite[Theorem 4]{amid2020reparameterizing}, which needed to explicitly assume convergence of the flow. Finally, existing results like \cite{woodworth2020kernel,wu2021implicit} change the initialization scale or a mirror map parameter to obtain $\ell_1$-regularization for very small and $\ell_2$-regularization for very large initialization/mirror map parameters, but do not guarantee $\ell_p$-minimization for $p \in (1,2)$. Theorem \ref{thm:PolynomialRegularizationIntroduction}, in contrast, covers all $\ell_p$-norms between $p=1$ and $p=2$ by varying the reparametrization $\sigmainn$. 
This is illustrated in Figure \ref{fig:Polynomial}, where one can observe that for $p=1.2$, respectively $p=1.8$, gradient descent with sufficiently small initialization converges towards the respective $\ell_p$-minimizers. \revFinal{When comparing different choices of $L$ in Figure \ref{fig:Polynomial_Varying_L}, one can see that $L$ only influences the convergence rate, not the limit of gradient descent. This is as predicted by our theory in Section \ref{sec:MainResults}, cf.\ Theorem \ref{thm:Convergence}.}

The second showcase result applies our insights to trigonometric reparametrizations.

\begin{theorem}
    \label{thm:TrigonometricRegularizationIntroduction}
    Under Assumption \ref{assumption:Intro}, let $\w \colon [0,T) \to \R^\N$ be any solution to \eqref{eq:gd_IRERM} with $\w_0 = \0$, $\A\in\mathbb{R}^{\M\times\N}$, and $\y\in \mathrm{im}(\A)$.
    Then, the following statements hold:
    \begin{enumerate}
        \item \revFinal{Assume} $\sigmainn(\z) = \sinh(\z)$. Then, $\w$ can be extended to a maximal solution $\w : [0,\infty) \rightarrow \R $, the limit $\wprodinfty := \lim_{\t \to \infty} \sigmainn(\w(\t)) \in \R^\N$ exists, and
        \begin{equation}
             \g_\text{sinh} \left( \wprodinfty \right)
             =
              \min_{\substack{\widetilde\z\in\R^\N,\\ \revFinal{L(\A\widetilde\z,\y)}=0}} \g_\text{sinh} \left( \widetilde\z \right),
        \end{equation}
        where
        \begin{align}\label{eq:TrigonometricRegularizationIntro_sinh}
            \g_\text{sinh} (\z) = \langle \z, \arctan(\z) \rangle - \langle \bo, \tfrac{1}{2} \log(\bo +\z^{\odot 2}) \rangle.
        \end{align}
        
        \item \revFinal{Assume} $\sigmainn(\z) = \tanh(\z)$. If there exists $\z \in \R_{\neq 0}^\N$ and $\A\z = \y$,\footnote{The set $\R_{\neq 0}^\N$ contains all vectors with non-zero coordinates. If $\A$ is generated at random with, e.g., iid Gaussian entries, this additional requirement holds with probability one.} $\w$ can be extended to a maximal solution $w : [0,\infty) \rightarrow \R $, the limit $\wprodinfty := \lim_{\t \to \infty} \sigmainn(\w(\t)) \in [-1,1]^\N$ exists, and
        \begin{equation}
             \g_\text{tanh} \left( \wprodinfty \right)
             =
              \min_{\substack{\widetilde\z\in [-1,1]^\N,\\ \revFinal{L(\A\widetilde\z,\y)}=0}} \g_\text{tanh} \left( \widetilde\z \right),
        \end{equation}
        where
        \begin{align}\label{eq:ImplicitRegularization_tanh}
            \g_\text{tanh} (\z) = \langle \z, \artanh(\z) \rangle.
        \end{align}
    \end{enumerate}
\end{theorem}


Theorem \ref{thm:TrigonometricRegularizationIntroduction} is a direct consequence of Corollaries \ref{cor:TrigonometricRegularization} and \ref{cor:TrigonometricRegularizationConvergence}. 
Whereas in Theorem \ref{thm:PolynomialRegularizationIntroduction} we could not set $\w_0 = \0$  since this is a saddle point of the loss function, in Theorem \ref{thm:TrigonometricRegularizationIntroduction} we can do so.
Observing furthermore that $\arctan (\ze) \approx \frac{\pi}{2} \text{sign} (\ze) $ for $\vert \ze \vert \gg 1$ and $ \arctan \left( \ze  \right)  \approx \ze  $ for $\vert \ze \vert \ll 1 $, we note that 
\begin{align}
    (\langle \z, \arctan (\z)\rangle)_i
    \approx
    \begin{cases}
        \frac{\pi}{2} \vert \ze_i \vert  &\text{if }\vert \ze_i \vert \gg 1\\
        \vert \ze_i \vert^2 &\text{if } \vert \ze_i \vert\ll 1 
    \end{cases}
    \label{eq:arctanhApprox}
\end{align}
which has the property that it behaves globally like the $\ell_1$-norm but is smooth around $0$. 
Since the additional $\log$-term in $\g_\text{sinh}$ is almost negligible in comparison, both regularizers $\g_\text{sinh}$ \revFinal{and $\g_\text{tanh}$ can be viewed as \emph{smooth} approximations} of the Huber loss without an additional hyper-parameter, cf.\ Figure \ref{fig:G1Huber_compare}. \revFinal{To the best of our knowledge, the two types of implicit regularization appearing in Theorem \ref{thm:TrigonometricRegularizationIntroduction} have not been discussed in existing literature. In light of the non-smoothness of the Huber loss, they might however prove useful in future research due to their smoothness.}

\begin{figure}[t!]
    \centering
    \includegraphics[width= 0.6\textwidth]{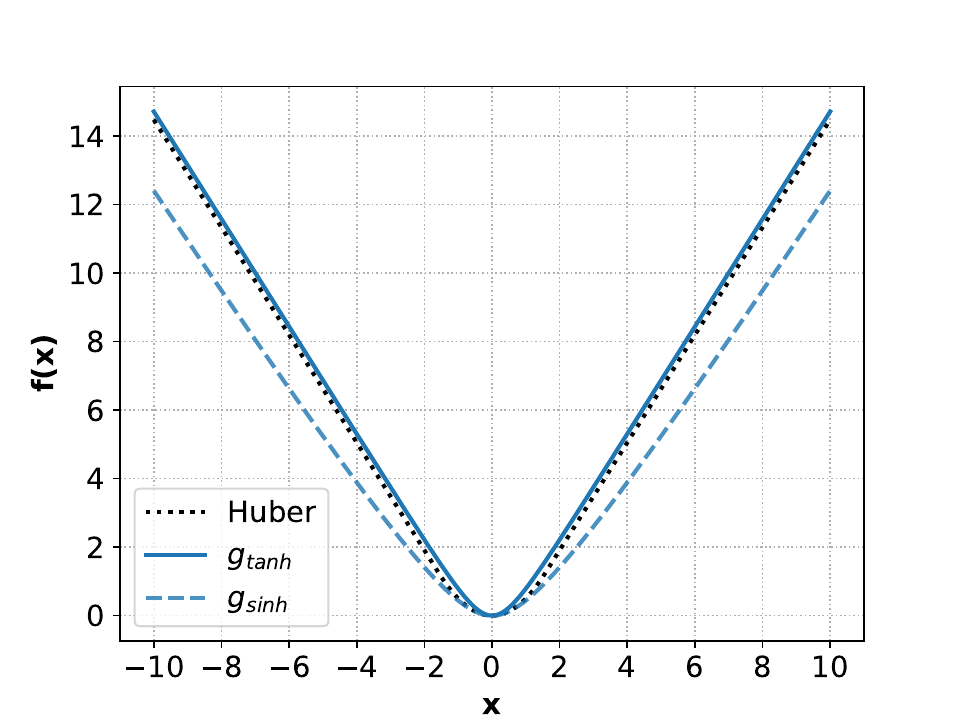}
    \caption{\revFinal{Comparison between $\g_\text{tanh}$, $\g_\text{sinh}$, and the Huber function \cite{Huber1964} with parameter $\frac{\pi}{2}$.}}
    \label{fig:G1Huber_compare}
\end{figure}

\begin{remark}
\label{rem:IRERM}
    The theory we present in the following is based on Definition \ref{definition:IRERM} which strongly restricts the choice of $\sigmainn$ and $\sigmaout$ to univariate functions. Let us mention that all results can be generalized in a straight-forward way to cover \emph{any} pair of entry-wise defined link and reparametrization functions, i.e., all functions $\sigmaout \colon \R^\M \to \R^\M$ and $\sigmainn \colon \R^\N \to \R^\N$ of the form
    \begin{align*}
        \sigmaout(\z) &= (\sigmaout^{(1)}(\ze_1),\dots , \sigmaout^{(\M)}(\ze_\M) )^T,\\
        \sigmainn(\z) &= (\sigmainn^{(1)}(\ze_1),\dots , \sigmainn^{(\N)}(\ze_\N) )^T,
    \end{align*}
    where $\sigmaout^{(1)},\dots,\sigmaout^{(\M)}, \sigmainn^{(1)},\dots, \sigmainn^{(\N)} \colon \R \to \R$ are real-valued functions. For the sake of conciseness, we refrain from presenting the results in full generality. Also note that the recent work \cite{Li2022implicit} allows even more general choices of $\sigmainn$ but requires twice continuous differentiability. 
\end{remark}

\subsection{Related work}
\label{sec:related_work}

Before turning to the mathematical details of our work, let us review the recent progress in the field.
Since modern machine learning models such as neural networks are typically highly overparameterized \cite{zhang17}, it is important to understand to which minimizer gradient descent converges.
\textit{Algorithmic/implicit regularization} refers to the phenomenon that gradient descent tends to pick minima with low intrinsic complexity even when no regularization is explicitly enforced.
Indeed, empirically it has been observed that the weights of trained (deep) neural networks are often approximately low-rank \cite{huh2023simplicitybias}.
However, a precise theoretical explanation does not seem to be within reach of current techniques.

Nonetheless, the deep learning theory community has made progress in understanding the implicit regularization phenomenon by studying simplified models.
For example, implicit regularization has been observed in classification tasks, see e.g., \cite{soudry2018implicit,nacson2019convergence,moroshko2020implicit,ji2019implicit,ji2021characterizing},
where gradient-based methods favor max-margin solutions.

For regression tasks, the implicit regularization on reparameterized loss functions has been studied in \cite{vaskevicius2019implicit,woodworth2020kernel,amid2020reparameterizing,amid2020winnowing,yun2021a,azulay2021implicit,li2021implicit,chou2021more}.
A key insight of this line of research is that gradient flow on the reparameterized function can be interpreted as mirror flow on the original function and that mirror flow converges to the solution which is closest to the initialization with respect to the Bregman divergence induced by the mirror map.
Recent work aims to characterize whether gradient flow for a particular reparameterization can be written as an equivalent mirror flow \cite{gunasekar2021mirrorless,Li2022implicit}. In a more general setting but under stronger assumptions on the reparametrization map (including higher order differentiability), the authors of \cite{Li2022implicit} derive convergence guarantees and a relation between reparametrization and implicit bias that is analogous to \eqref{eq:hprimedefinitionWithoutPositivity} below.

Building on these insights in linear regression, the authors of \cite{chou2022non} use the reparametrization in \eqref{eq:Lover_simplified} to encode convex constraints into the optimization landscape. In \cite{kolb2023smoothing}, the authors discuss how to use Hadamard product reparametrization to replace non-smooth $\ell_{p,q}$-regularized problems by smooth $\ell_2$-regularized counterparts. In the context of sparse phase retrieval, mirror flow with hypentropy mirror map and the closely related quadratically reparameterized Wirtinger flow have been studied in \cite{wu2020continuous,wu2023nearly}. %

Moving beyond vector-valued problems, implicit regularization has also been studied in overparameterized low-rank matrix recovery.
In the influential work \cite{gunasekar2017implicit} it has been observed that factorized gradient descent converges to a solution with low-rank when the initialization is chosen sufficently small. %
A number of works investigated both the symmetric setting \cite{li2018algorithmic,litowards,stoger2021small,xu2023power,ding2022validation,jin2023understanding} and the asymmetric setting \cite{soltanolkotabi2023implicit}.
In \cite{arora2019implicit,chou2020implicit} the authors studied the scenario with more than two factors, which sometimes is referred to as \textit{deep matrix factorization}.
In \cite{wu2021implicit}, it has been shown that mirror descent with sufficiently small initialization converges to the solution with lowest nuclear norm.
In \cite{razin2020implicit} an example was presentend, in which factorized gradient descent with small initialization converges to a low-rank solution which is not the nuclear norm minimizer.
In the setting of low-rank tensor recovery, \cite{Razin2021implicitTensor} investigated the bias of gradient descent towards low-rank tensors.

This paper, along with most of the aforementioned works, focuses on how implicit regularization is influenced by a combination of reparametrization and scale of initialization.
It is worth to note that there are other mechanisms which induce algorithmic regularization such as weight normalization \cite{chou2023robust} or label noise \cite{Pesme2021implicit,vivien2022label}. In  \cite{andriushchenko2022sgd,even2023s} an intriguing connection regarding implicit regularization induced by large step sizes coupled with SGD noise has been discussed.

\subsection{Outline and Notation}

The remaining paper is structured as follows. In Section \ref{sec:MainResults}, we will present our results in full generality and discuss how to derive Theorems \ref{thm:PolynomialRegularizationIntroduction} and \ref{thm:TrigonometricRegularizationIntroduction} from them. Parts of the proofs are deferred to Section \ref{sec:Proof}.

We briefly introduce the notation we use in the rest of the work. For $\N\in\mathbb{N}$, we denote $[\N]=\{1,2,\dots,N\}$. We abbreviate the sets of non-negative, positive, and non-zero real numbers by $\R_{\ge 0}$, $\R_{>0}$, and $\R_{\neq 0}$, respectively. Boldface lower-case letters such as $\z$ represent vectors with entries $\ze_n$, while boldface upper-case letters such as $\A$ represent matrices with entries $\Ae_{mn}$. We denote by $\sigma_{\min}(\A)$ the smallest non-zero singular value of $\A$. For vectors $\u,\v \in\mathbb{R}^\N$, $\u\geq\v$ means that $u_\n\geq v_\n$ for all $\n\in[\N]$. In addition, $\u \geq 0$ means that $u_\n\geq 0$ for all $\n\in[\N]$. The operator $\mathrm{diag}$ can be applied to matrices and vectors. In the former case, it extracts the diagonal of the input as a vector; in the latter case, it embeds the input as diagonal into a square matrix. We will use $\0$ and $\bo$ to denote the all-zero and all-ones vector. The dimensions of both will always be clear from the context.\\

\textbf{Reparametrized quantities:} Since our work considers reparametrizations of the input space $\R^\N$, we will use tilde-symbols to distinguish a plain vector $\z \in \R^\N$ from its reparametrization $\widetilde\z = \sigmainn(\z) \in \R^N$, which is usually only compared to other reparametrized vectors. In a similar way, the reparametrized version of the gradient flow $\w(\t)$ with domain $\R^\N$ will be abbreviated by $\wprod(\t) = \sigmainn(\w(\t))$ with (reparametrized) domain $\wDomainRP \subset \R^\N$.\\

\textbf{Hadamard calculus:} We use $\odot$ to denote Hadamard products and powers, e.g., $(\z\odot\y)_n = \ze_n\ye_n$ and $(\z^{\odot\L})_n = \ze_n^{\L}$. We denote the closed positive orthant by $\R_{\ge 0}^\N$. For any $S \subset \R^\N$,  $\mathrm{int}(S)$ is the interior of $S$ and $\partial S$ the boundary of $S$. If $\h:\mathbb{R}\to\mathbb{R}$ is a real-valued function, it acts entry-wise on vectors, i.e., $\h(\z) \in \mathbb{R}^{\N}$ with $[\h(\z)]_{\n} = \h(\ze_n)$. If $\h:\mathbb{R}^{\N}\to\mathbb{R}^{\N}$ is a decoupled function, i.e.
\begin{equation*}
    \h(\z) = (\h_1(\ze_1),\dots , \h_\N(\ze_\N) )
\end{equation*}
for functions $\h_1, \h_2, \ldots, \h_N: \R \to \R$, we write its entry-wise derivative as $\h':\mathbb{R}^{\N}\to\mathbb{R}^{\N}$ where
\begin{equation*}
    \h'(\z) = (\h_1'(\ze_1),\dots , \h_\N'(\ze_\N) ).
\end{equation*}

\section{Reparametrizing gradient flow on \revFinal{generalized linear models (GLMs)}}
\label{sec:MainResults}

In this section, we present our results in full detail and explain how to derive Theorems \ref{thm:PolynomialRegularizationIntroduction} and \ref{thm:TrigonometricRegularizationIntroduction} from them. All omitted proofs can be found in Section \ref{sec:Proof} below.

\subsection{General framework and main result}
First, we describe our general framework and state our main result in full generality.

\begin{assumption}
\label{assumption:Main}
    Let $\A\in\mathbb{R}^{\M\times\N}$ and $\y\in\mathbb{R}^{\M}$. We will use the following set of assumptions.
    \begin{enumerate}[label=(A\arabic*)]
        \item \underline{Properties of $\sigmaout$}: \revFinal{The function} $\sigmaout \colon \R \to \R$ is continuously differentiable and injective, and $\y \in  \sigmaout (\R^M)$.
        \label{as:A1}
        \item \label{as:A2}\underline{Properties of $\sigmainn$}: \revFinal{There exists an interval $\wDomainRPOneD \subset \R$ such that $\sigmainn \colon \R \to \wDomainRPOneD$ is invertible and continuously differentiable, and 
        $\sigmainn' (x) \neq 0$ almost everywhere.} \revFinal{Furthermore,
        \begin{equation*}
        ([\sigmainn^{-1}]')^2
        \in L_1^c(\wDomainRPOneD) 
        = \{f:  f\vert_C \in L_1(C) \text{ for any compact set } C \subset \wDomainRPOneD \},        
        \end{equation*}
        where $([\sigmainn^{-1}]')^2$ is defined almost everywhere on $\wDomainRPOneD$.}
        
        \item \underline{Properties of $L$}: $L \colon \R^\M \times \R^\M \to \R$ is continuously differentiable, convex in the first argument, and satisfies $\min L = 0$ with $L(\z,\z') = \min_{\z,\z' \in \R^\M} L(\z,\z')$ if and only if $\z = \z'$.
        \label{as:A3}
        \item There exists $\widetilde \z_0 \in \wDomainRP$ with $\revFinal{\sigmaout(\A\widetilde \z_0)} = \y$.\footnote{Note that \revFinal{under Assumptions \ref{as:A2} and \ref{as:A3} this assumption is equivalent to $\mathcal{L}$ realizing} its minimum at zero. 
        }
        \label{as:B2}
    \end{enumerate}
    Parts of the results will require the additional assumption that
    \begin{enumerate}[label=(A\arabic*)]
        \setcounter{enumi}{4}
        \item \revFinal{The function $\v \mapsto L(\sigmaout(\v),\y)$ is convex.}
        \label{as:B1}
    \end{enumerate}
\end{assumption}

\begin{remark}
\label{rem:Assumptions}
    \revFinal{In this paper we will apply our results only to reparametrization functions $\sigmainn \colon \R \to \R$ with $\sigmainn'(x) = 0$ for at most finitely many points.  Nevertheless, a general gradient flow trajectory $\w \colon [0,T] \to \R^N$ on \eqref{eq:IRERM} can hit such degenerate points of $([\sigmainn^{-1}]')^2$, uncountably many times, prohibiting integrability of the composite function $([\sigmainn^{-1}]')^2 \circ \w$ which is needed in our analysis. Our restriction in the following results to regular flows for which $|T_\circ| := |\{ t \in [0,T] \colon \sigmainn'(w(t)_i) = 0, \text{ for } i \in [N] \}| < \infty$ (see Definition \ref{definition:IRERM}) is thus critical in all settings in which $\sigmainn'(x) = 0$ for at least one point $x \in \R$. As argued in Remark \ref{rem:MainDefinition} (iii), this does not diminish the impact of our results for understanding the implicit bias of gradient descent.}
\end{remark}

Assuming convergence of \eqref{eq:gd_IRERM} to global optimality in \eqref{eq:IRERM}, Theorem \ref{theorem:vector_IRERM} below characterizes to which minimizer the flow converges. In particular, among all global minimizers of $\IRERM$ the limit of gradient flow minimizes the distance to the initialization $\w_0$ measured in terms of a Bregman divergence that only depends on the choice of $\sigmainn$. Let us briefly recall the definition of Bregman divergences.

\begin{definition}[Bregman Divergence]\label{def:Bregman_Divergence}
Let $\F:\Omega\to\mathbb{R}$ be a continuously differentiable, strictly convex function defined on a closed convex set $\Omega$.
The Bregman divergence associated with $\F$ for points $p\in\Omega$ and $ q \in \mathrm{int} \left( \Omega \right) $ is defined as
\begin{equation}
    D_{\F}(p,q) = \F(p) - \F(q) - \langle \nabla \F(q), p-q \rangle.
\end{equation}
\end{definition}

Theorem \ref{theorem:vector_IRERM} only requires that the reparametrization $\sigmainn$ has one continuous derivative whereas existing work \cite{Li2022implicit} requires that $\sigmainn$ has two continuous derivatives. 

\begin{theorem}[Implicit Bias]\label{theorem:vector_IRERM}
    Let $\w: [0,T) \to \mathbb{R}^N$ with $T \in \R_{>0} \cup \{+\infty\}$ be any solution to the differential equation \eqref{eq:gd_IRERM} in Definition \ref{definition:IRERM}, for $\A\in\mathbb{R}^{\M\times\N}$, $\y\in\mathbb{R}^{\M}$, $\sigmainn \colon \R \to \R$, $\sigmaout \colon \R \to \R$, and $L \colon \R^\M \times \R^\M \to \R$.
    Define the reparametrized flow and loss-function
    \begin{align*}
        \wprod (t) := \sigmainn(\w (t))
        \quad
        \text{and} \quad \LossM(\widetilde{\z}) = L\big( \sigmaout(\A \widetilde{\z}), \y \big),
    \end{align*}
    for all $\widetilde{\z} \in \mathbb{R}^{N}$. \revFinal{Let Assumptions \ref{as:A1} -- \ref{as:B2} hold} and
    assume that $\w$ is regular.
    
    We denote by $\h \colon \wDomainRPOneD \to \R$ a continuous antiderivative of
    \begin{align}\label{eq:hprimedefinitionWithoutPositivity}
        \h'(\widetilde z ) 
        := \big( [\sigmainn^{-1}]'(\widetilde z) \big)^2,
    \end{align}
    and we denote by $\H \colon \wDomainRPOneD \to \R$ a continuous antiderivative of $h$.\footnote{By \ref{as:A2}, we know that $h'$ is integrable and thus has a continuous antiderivative. \label{footnote:h}}
    If the limit $\wT:=\lim_{\t\to T}\w(\t)$ exists and satisfies $\mathcal{L}(\wT) = 0$, and if at least one of the following two conditions holds
        \begin{itemize}
            \item $\wprodT \in \wDomainRP$ or
            \item the continuous extensions of $\h$ and $\H$ to $\overline{\wDomainRPOneD}$ are well-defined and finite-valued,
        \end{itemize}
        then
        \begin{equation}\label{eq:implicit_regularization_Bregman}
            \wprodT \in \argmin_{\widetilde\z\in\overline{\wDomainRP}, \LossM(\widetilde\z) = 0 %
            } \langle \bo, \H(\widetilde\z) - \widetilde\z\odot\h(\wprod_0) \rangle.
        \end{equation}
        
        Denote by $D_{F}(\widetilde\z,\wprod_0)$ the Bregman divergence defined on $\wDomainRP$ with respect to $F(\widetilde\z) := \langle \bo, \H(\widetilde\z)\rangle$ and extended to $\overline{\wDomainRP}$ by continuity (including $+\infty$ as admissible objective value).\footnote{\label{footnote:H} It is easy to verify in Definition \ref{def:Bregman_Divergence} that $D_F$ is invariant under adding affine functions to $F$ and thus independent of the particular choices of $\h$ and $\H$.}
        Then minimizing the objective in \eqref{eq:implicit_regularization_Bregman} is equivalent to minimizing the Bregman divergence $D_{F}(\widetilde\z,\wprod_0)$ in the sense that the set of global minimizers is the same.
        
\end{theorem}

The proof of Theorem \ref{theorem:vector_IRERM} is provided in Section \ref{sec:Proofvector_IRERM}. The specific shape of $\sigmainn$ assumed in \ref{as:A2} is crucial for allowing a concise representation of $H''= h'$.
Let us highlight some more points before moving on:
\begin{enumerate}[label=(\roman*)]

    \item {\bf Dependence on initialization $\wprod_0$}: 
    The implicit regularization in \eqref{eq:implicit_regularization_Bregman} not only
    depends on the choice of $\sigmainn$ but also on the initialization $\wprod_0$. Theorems \ref{thm:PolynomialRegularizationIntroduction} and \ref{thm:TrigonometricRegularizationIntroduction} assumed that $\wprod_0$ is close to the origin.
    Note, however, that other initialization regimes can be considered, cf.\ \cite{woodworth2020kernel}. 
    \item {\bf Strictly convex loss functions:}
    If $\mathcal L$ is strictly convex, one can remove the assumption that $\mathcal L(\revFinal{\wT}) = 0$ and replace \eqref{eq:implicit_regularization_Bregman} by
    \begin{align*}
        \wprodT \in \argmin_{\widetilde\z\in\overline{\wDomainRP}, \LossM(\widetilde\z) = \LossM(\wprodT) %
        } \langle \bo, \H(\widetilde\z) - \widetilde\z\odot\h(\wprod_0) \rangle,
    \end{align*}
    i.e., Theorem \ref{theorem:vector_IRERM} also applies to settings in which the loss function is not globally minimized.
    
    \item {\bf More general reparametrizations:} To cover the more general setting described in Remark \ref{rem:IRERM}, in which $\sigmainn \colon \R^N \to \R^N$ may vary on different entries of input vectors, one only has to replace $\wDomainRP$ in Theorem \ref{theorem:vector_IRERM} with $\wDomainRP = (\wDomainRPOneD)_1 \times \cdots \times (\wDomainRPOneD)_N$, for $ (\wDomainRPOneD)_1, \dots, (\wDomainRPOneD)_N \subset \R$, require that Assumption \ref{as:A2} holds for each marginal of $\sigmainn$, and define $h \colon \wDomainRP \to \R^N$ accordingly. 
    However, to keep the presentation concise, we will restrict our presentation to the scenario where all marginals are the same.
\end{enumerate}

Whereas $\sigmainn$ heavily influences the characterization of $\wprodinfty$ in \eqref{eq:implicit_regularization_Bregman} via $F$, the specific choices of $\sigmaout$ and $L$ play no role, i.e., the implicit bias of gradient flow on \eqref{eq:IRERM} is solely determined by $\sigmainn$ and $\w_0$.
The following result however shows that $\sigmaout$ and $L$ strongly influence the convergence behavior of gradient flow.
In particular, \revFinal{if $\v \mapsto L(\sigmaout(\v),\y)$ satisfies} the so-called \emph{Polyak-Lojasiewicz inequality} we derive a linear rate.

\begin{definition}\label{def:polyak}
    \cite[Condition C]{polyak1963gradient}
    A differentiable function $f \colon \R^\N \to \R$ with minimum $f_\star$ satisfies the  \emph{Polyak-Lojasiewicz inequality} with $\mu > 0$ if, for any $\z \in \R^\N$, it holds that
    \begin{align*}
        \| \nabla f(\z) \|_2^2
        \ge 2\mu (f(\z) - f_\star).
    \end{align*}
\end{definition}

\begin{remark} \label{rem:PL}
    Let us recall two useful facts: 
    \begin{enumerate}
        \item[(i)] Any $\mu$-strongly convex function $f$, i.e., for any $\u,\v \in \R^\N$,
        \begin{align*}
            f(\v) \ge f(\u) + \langle \nabla f(\u), \v-\u \rangle + \frac{\mu}{2} \| \u - \v \|_2^2,
        \end{align*}
        satisfies the Polyak-Lojasiewicz inequality with $\mu$.
        \item[(ii)] If $f(\z) = g(\A\z)$, for some matrix $\A~\in~\R^{\M \times \N}$ and $g \colon \R^\M \to \R$ satisfying the Polyak-Lojasiewicz inequality with $\mu$, then $f$ satisfies the Polyak-Lojasiewicz inequality with $(\sigma^2_{\min}(\A)\;\mu)$, 
        where $\sigma_{\min} \left(\A\right)$ denotes the smallest non-zero singular value of $\A$. \revFinal{This includes linear regression as a special case.}
    \end{enumerate}
\end{remark}

\begin{theorem}[Convergence] 
	\label{thm:Convergence}
	 Let $\w: [0,T) \to \mathbb{R}^N$ with $T \in \R_{>0} \cup \{+\infty\}$ be any solution to the differential equation \eqref{eq:gd_IRERM} in Definition \ref{definition:IRERM}, for $\A\in\mathbb{R}^{\M\times\N}$, $\y\in\mathbb{R}^{\M}$, $\sigmainn \colon \R \to \R$, $\sigmaout \colon \R \to \R$, and $L \colon \R^\M \times \R^\M \to \R$.
	Define the reparametrized flow $\wprod$ and loss function $\LossM$ via
	\begin{align*}
		\wprod (t) := \sigmainn(\w (t))
		\quad
		\text{and} \quad \LossM(\widetilde{\z}) = L\big( \sigmaout(\A \widetilde{\z}), \y \big)
	\end{align*}
	for all $\widetilde{\z} \in \mathbb{R}^{N}$. 
 \revFinal{Let Assumptions \ref{as:A1} -- \ref{as:B1}} hold, and 
 assume that $\w$ is regular.
 
 Recall from Theorem \ref{theorem:vector_IRERM} the function $h$ defined in \eqref{eq:hprimedefinitionWithoutPositivity} and the Bregman divergence $D_F$.
	Then, the following claims hold:
	
	\begin{enumerate}[label=(\roman*)]
		\item \underline{General convergence rate:} %
			For all $t \in [0,T)$, it holds that 
			\begin{equation*}
				\LossM(\wprod(\t)) 
				\le 
				\frac{D_F \left( \zprod_0 , \wprod_0  \right)}
				{t},
			\end{equation*}
			where $\zprod_0$
            \revFinal{is chosen as in}
            in Assumption \ref{as:B2} and $D_F \left( \zprod_0 , \wprod_0  \right) < \infty$. If $T = \infty$, this implies that $\lim_{t \to \infty} \LossM(\wprod(\t)) = 0$.
			\label{thm:Convergence_i}
			\item \underline{Linear convergence rate:} If, in addition, \revFinal{$\v \mapsto L(\sigmaout(\v),\y)$} satisfies the Polyak-Lojasiewicz inequality for $\mu>0$ and if there exists $r > 0$ such that $(\h')^{\odot -1}(\wprod(\t)) \ge r$, for \revFinal{every $\t \in [\t_0,\t_1)$ where $\t_0,\t_1 \in [0,T) \cup \{\infty\}$ satisfy that $[\t_0,\t_1) \cap T_\circ = \emptyset$ and $T_\circ$ is defined as in Remark \ref{rem:Assumptions},} 
            then
			\begin{align*}
				\LossM(\wprod(\t)) \le \revFinal{\LossM(\wprod(\t_0)) e^{-2r\mu (\t-\t_0)}}.
			\end{align*}
            \label{thm:Convergence_ii}
			\item \underline{Extending a bounded trajectory:} 
            If $\underset{t \in [0,T)}{\sup} \| \wprod(\t) \|_2  < + \infty$, then $\wprodT := \lim_{t \to T} \wprod(t)$ exists. Moreover, either $\wprodT \in \partial \wDomainRP$ and the trajectory is maximal
            or there exists $\varepsilon > 0$ such that $\w$ can be extended to a solution 
            $\w \colon [0,T+\varepsilon) \to \R^N$ of \eqref{eq:gd_IRERM}.
			\label{thm:Convergence_iii}
            
			\item \underline{Convergence of bounded trajectory:}
			If $T = \infty$ and $\underset{t \ge 0}{\sup} \| \wprod(\t) \|_2  < + \infty$,  
			then
			\begin{align}
				\label{eq:LimExists}                
				\wprodinfty :=  \underset{t \rightarrow + \infty}{\lim} \wprod(\t)   \in \overline{\wDomainRP}
				\qquad \text{exists with} \qquad \LossM(\wprodinfty) = 0.                
			\end{align}
		    \label{thm:Convergence_iv}
     \end{enumerate}
     If $\wDomainRPOneD$ is unbounded, which means that boundedness of the trajectory $\wprod(t)$ does not trivially hold, we furthermore have the following:
     \begin{enumerate}[resume,label=(\roman*)]
		    \item \underline{Boundedness of trajectory:}			
                If the Bregman balls $B_{D_F,r}(\x) = \{ \z \in \wDomainRP \colon D_F(\x,\z) \le r \}$ are bounded 
                with respect to the Euclidean norm $\Vert \cdot \Vert_2 $, 
                for any $\x \in \wDomainRP$ and $r \ge 0$, then $$\underset{t \in [0,T)}{\sup} \| \wprod(\t) \|_2  < + \infty.$$           
			In particular, the Bregman balls $B_{D_F,r}(\x)$ are bounded if the univariate function $\widetilde z \mapsto \widetilde z \cdot h'(\widetilde z)$ is in the function class
			\begin{align}
				\label{eq:F_D}
				\mathcal F_\wDomainRPOneD = \Big\{ f \in L_1^c(\wDomainRPOneD) \colon \Big| \int_{z}^b f(x) dx \Big| = \infty, \text{ for any } z \in \wDomainRPOneD  \text{ and } b \in \wDomainRPlim \Big\},
			\end{align}
			where $\wDomainRPlim := \{ \inf \wDomainRPOneD, \sup \wDomainRPOneD \} \cap \{-\infty, \infty\}$.
			\label{thm:Convergence_v}
		\end{enumerate}
	\end{theorem}
\revFinal{The proof of Theorem \ref{thm:Convergence} is provided in Section \ref{sec:ProofConvergence}.
The result sheds some light on the influence that the different components of our model in \eqref{eq:IRERM} have on the expected convergence behaviour of gradient flow. 
First, the objective value of $\LossM$ decays along the trajectory with rate $\mathcal{O}(\tfrac{1}{\t})$ independent of whether $\wprod$ converges or not. Second, if the composition $\v \mapsto L(\sigmaout(\v),\y)$ satisfies the Polyak-Lojasiewicz condition, this rate is improved to $\mathcal{O}(e^{-c\t})$ on non-critical time intervals. 
Third, the choice of $\sigmainn$ influences the convergence in several ways: it affects the quantities $D_F \left( \zprod_0 , \wprod_0  \right)$ and $r$ in Theorem \ref{thm:Convergence} \ref{thm:Convergence_i} and \ref{thm:Convergence_ii} via $F$ and $h$, and it determines whether $[z \mapsto \widetilde z \cdot h'(\widetilde z)] \in \mathcal F_\wDomainRPOneD$ and thus whether one can expect boundedness/convergence of the trajectory $\wprod$. Whereas an analog to Claim \ref{thm:Convergence_i} for mirror descent with spectral hypentropy mirror map previously appeared in \cite{wu2021implicit} and analogs to Claims \ref{thm:Convergence_i} and \ref{thm:Convergence_ii} were observed for mirror flow on strictly convex objectives in \cite{tzen2023variational}, the sufficient condition \eqref{eq:F_D} for boundedness of the trajectory $\wprod$ is novel to the best of our knowledge.
}

\revFinal{
\begin{remark}
    \begin{enumerate}
        \item[(i)] Note that $\sigmaout = \mathrm{Id}$ satisfies Assumptions \ref{as:A1} and \ref{as:B1}, i.e., Theorem \ref{thm:Convergence} applies to reparametrized linear regression, cf.\ Remark \ref{rem:PL} (ii).
        \item[(ii)] The linear convergence rate in Theorem \ref{thm:Convergence} \ref{thm:Convergence_ii} only applies to time intervals that are separated from the set of critical times $T_\circ$. Indeed, since $(\h')^{\odot -1}(\wprod(\t))$ vanishes for $\t \in T_\circ$, we could not derive strong convergence guarantees close to $T_\circ$. This is consistent with the fact that at such points the dynamics are non-unique and may become stationary.
    \end{enumerate}
\end{remark}
}

\subsection{How to induce specific regularization}
\label{sec:Example}

With Theorem \ref{theorem:vector_IRERM} and Theorem \ref{thm:Convergence} in hand, we now discuss how the specific instances of implicit regularization in Theorem \ref{thm:PolynomialRegularizationIntroduction} and Theorem \ref{thm:TrigonometricRegularizationIntroduction} can be derived from them. We restrict ourselves to the case that $\wprod_0 = \alpha \bo$, for $\alpha > 0$ being small, a common setting studied in many related works, e.g. \cite{arora2018optimization,azulay2021implicit,chou2021more,even2023s}
Let us begin with the setting of Theorem \ref{thm:PolynomialRegularizationIntroduction} and apply Theorem \ref{theorem:vector_IRERM}.

\revFinal{
We consider polynomial and (hyperbolic) trigonometric reparametrizations as representative examples illustrating different types of induced regularization. The polynomial family provides a continuum of behaviors that interpolate between distinct regimes, while remaining analytically tractable. In contrast, trigonometric reparametrizations lead to qualitatively different regularization structures and illustrate how non-polynomial nonlinearities can be incorporated within the same framework \revFinal{by using our techniques}.}

\begin{corollary}[Polynomial Regularization]
\label{cor:PolynomialRegularization}
    Let $\sigmainn(z) = \sign(z) |z|^{\frac{2}{p}}$, for $p \in (1,2)$, and let $\w: [0,T) \to \mathbb{R}^N$ with $T \in \R_{>0} \cup \{+\infty\}$ be any 
    regular
    solution to the differential equation \eqref{eq:gd_IRERM} in Definition \ref{definition:IRERM}, for $\A\in\mathbb{R}^{\M\times\N}$, $\y\in\mathbb{R}^{\M}$ with $\sigmaout \colon \R \to \R$ and $L \colon \R^\M \times \R^\M \to \R$
    satisfying \revFinal{Assumptions \ref{as:A1}, \ref{as:A3}, and \ref{as:B2}}.     
    Define $\wprod_0 = \sigmainn(\w_0) = \alpha \bo$, for $\alpha > 0$ and 
    set 
    \begin{equation*}
        \mu 
        :=
        \min_{\widetilde\z\in \R^\N, \LossM(\widetilde\z) = 0} \|\widetilde\z\|_{\p}.
    \end{equation*}
    If $\wprod_{T,\alpha} := \lim_{\t \to T} \sigmainn(\w(\t))$ with $\LossM(\wprod_{T,\alpha}) = 0$ exists and
    \begin{equation*}
        \alpha
        \le \frac{1}{2^{1/(p-1)} N^{1/p}} \mu,
    \end{equation*}
    then it holds that
    \begin{equation}
    \label{eq:PolynomialRegularizationLpmin}         
        \| \wprod_{T,\alpha} \|_p
        \le
        \left(1+ 4 N^{1-1/p} \left( \frac{ \alpha }{ \mu }\right)^{p-1}    \right) \cdot \mu.         
    \end{equation}
    
\end{corollary}

For the sake of conciseness, we defer the proof of Corollary \ref{cor:PolynomialRegularization} to Section \ref{sec:Non-asymptoticBounds} and only provide a sketch of the argument here. \revFinal{The assumption that $\widetilde\w_0$ is aligned with the vector of ones is made for convenience, as it ensures that all coordinates evolve \revFinal{identically} and reduces the analysis \revFinal{of the Bregman divergence} to a simplified expression. 
}

\begin{proof}[Proof sketch:]
    To use Theorem \ref{theorem:vector_IRERM}, we just have to verify Assumption \ref{as:A2} for $\wDomainRP = \R^\N$ since $\wprod_{T,\alpha} \in \wDomainRP$ by existence of the limit, i.e., the additional conditions of Theorem \ref{theorem:vector_IRERM} are automatically satisfied. For $p \in (1,2)$, the function $\sigmainn$ is continuously differentiable and invertible, and the derivative $\sigmainn'(z) = \frac{2}{p} |z|^{\frac{2-p}{p}}$ vanishes only for $z = 0$.
    Consequently, the function
    \begin{align}
    \label{eq:hPrime}
    	[\sigmainn^{-1}]'(\widetilde{z})^2 = \frac{p^2}{4} |\widetilde{z}|^{p-2},
    \end{align}
    is defined on $\R^\N \setminus \{ \boldsymbol{0} \}$ and is in $L_1^c(\R^\N)$.
    Now applying Theorem \ref{theorem:vector_IRERM} yields that 
    \begin{align}\label{eq:implicit_regularization}
        \wprod_{T, \alpha} \in \argmin_{\widetilde\z\in\overline{\wDomainRP}, \LossM(\widetilde\z) = 0} \langle \bo, \H(\widetilde\z) - \widetilde\z\odot\h(\wprod_0) \rangle, 
    \end{align}
    where $\H$ and $\h$ have been defined in the theorem statement.
    One computes that
    \begin{align}
    \label{eq:fw0_Polynomial}
        \langle \bo, \H(\widetilde\z) - \widetilde\z\odot\h(\wprod_0) \rangle 
        = 
            \frac{p}{4(p-1)} \big( \langle \bo, |\widetilde \z|^{\odot p} \rangle - p \langle \bo, \widetilde \z \odot \wprod_0^{\odot (p-1)} \rangle \big).
    \end{align}
    Let us now \emph{informally} take the limit $\alpha \to 0$ for $\wprod_0 = \alpha \bo$. 
    Since $\p>1$, the factor $\frac{\p}{\p-1}$ is positive and the term $\langle \bo, |\widetilde \z|^{\odot p} \rangle$ dominates. Hence, \eqref{eq:implicit_regularization} becomes equivalent to minimizing the $\ell_p$-norm over all global minimizers of $\LossM$.
\end{proof}

By setting $p = \frac{2}{L}$, Corollary \ref{cor:PolynomialRegularization} \revFinal{complements} the results of \cite{chou2021more} from $L \ge 2$ to $L \in (1,2)$, cf.\ Equation \eqref{eq:Lover_simplified} above.
To obtain Theorem \ref{thm:PolynomialRegularizationIntroduction} in full, we still need to apply Theorem \ref{thm:Convergence} to $\sigmainn(z) = \sign(z) |z|^{\frac{2}{p}}$.

\begin{corollary}[Polynomial Regularization --- Convergence]
\label{cor:PolynomialRegularizationConvergenceRate}
    Let $\sigmainn(z) = \sign(z) |z|^{\frac{2}{p}}$, for $p \in (1,2)$, and let $\w: [0,T) \to \mathbb{R}^N$ with $T \in \R_{>0} \cup \{+\infty\}$ be any 
    regular
    solution to the differential equation \eqref{eq:gd_IRERM} in Definition \ref{definition:IRERM}, \revFinal{for $\A\in\mathbb{R}^{\M\times\N}$, $\y\in\mathbb{R}^{\M}$, $\sigmaout \colon \R \to \R$, and $L \colon \R^\M \times \R^\M \to \R$ satisfying Assumptions \ref{as:A1}, \ref{as:A3}, \ref{as:B2}, and \ref{as:B1}}.
    Then the following claims hold:
    \begin{itemize}
    \item If $T<+ \infty$, 
    then the limit $ \lim_{t \rightarrow T}  \w (t) $ exists and $\w$ is not a maximal solution.
    \item 
    Suppose that either $T=+\infty$, or $T<+\infty$ and $\w$ can be extended to a regular solution $\w \colon [0,\infty) \to \R^N$ (without relabeling the flow).
    Then
    $\w \colon [0,\infty) \to \R^N$ converges to  $\winfty \in \R^\N$ with  $\LossM (\winfty) = 0$ and with the rate described in Theorem \ref{thm:Convergence} \ref{thm:Convergence_i}.
    
    \end{itemize}
\end{corollary}

\begin{proof}
    To apply Theorem \ref{thm:Convergence}, we only need to verify Assumption \ref{as:A2} for $\wDomainRP = \R^\N$. This has already been done in the proof of Corollary \ref{cor:PolynomialRegularization}. 
    Since we assume that $\w$ is regular we can apply Theorem \ref{thm:Convergence}.
    
    We first check that the function $\widetilde z \mapsto \widetilde z \cdot h'(\widetilde z)$ lies in the function class $\mathcal{F}_\wDomainRPOneD$, where
    \begin{align*}
        h'(\widetilde z) = \big( [\sigmainn^{-1}]'(\widetilde z) \big)^2 = \frac{p^2}{4} |\widetilde z|^{p-2}
    \end{align*}
    and $\mathcal{F}_\wDomainRPOneD$ was defined in \eqref{eq:F_D}. For $p \in (1,2)$, one has by continuity that  $\widetilde z \mapsto \widetilde z \cdot h'(\widetilde z) = \frac{p^2}{4} \sign(\widetilde z) |\widetilde z|^{p-1} \in L_1^c(\wDomainRPOneD\revFinal{\setminus \{0\}})$. Moreover, for $p \in (1,2)$ and any $\widetilde z \in \R$,
    \begin{align*}
        \int_{\widetilde z}^\infty \frac{p^2}{4} \sign(\tau) |\tau|^{p-1} d\tau = \infty
        \quad \text{and} \quad 
        \int_{\widetilde z}^{-\infty} \frac{p^2}{4} \sign(\tau) |\tau|^{p-1} d\tau = -\infty
    \end{align*} 
    such that $\widetilde z \mapsto \widetilde z \cdot h'(\widetilde z) \in \mathcal{F}_\wDomainRPOneD$.
    
    Theorem \ref{thm:Convergence} \ref{thm:Convergence_v} now yields that $\underset{t \in [0,T)}{\sup} \| \w(\t) \|_2  < + \infty$.
    Thus, the first claim holds by Theorem \ref{thm:Convergence} \ref{thm:Convergence_iii} since $ \partial \wDomainRP = \emptyset $.
    To show the second claim we first note that as before we can argue that the (possibly extended) function $ \w \colon [0,\infty) \to \R^N $ has the property that $\underset{t \in [0,\infty)}{\sup} \| \w(\t) \|_2  < + \infty  $. 
    Then the claim follows from Theorem \ref{thm:Convergence} \ref{thm:Convergence_iv}.
    This concludes the proof.
\end{proof}

After having derived Corollaries \ref{cor:PolynomialRegularization} and \ref{cor:PolynomialRegularizationConvergenceRate}, which induce Theorem \ref{thm:PolynomialRegularizationIntroduction}, let us now turn to the analogous results for trigonometric reparametrizations as considered in Theorem \ref{thm:TrigonometricRegularizationIntroduction}. In contrast to Corollary \ref{cor:PolynomialRegularization}, we set the initialization magnitude $\alpha$ directly to zero in order to simplify the statement.\footnote{In the polynomial case this was not possible \revFinal{since Theorem~\ref{theorem:vector_IRERM} applies only to trajectories whose limit is a global minimizer of the loss. In the polynomial case, initializing exactly at a saddle point (e.g., $\wprod_0 = \mathbf{0}$) yields a stationary trajectory that does not converge to a global minimizer.}}

\begin{corollary}[Trigonometric Regularization]
\label{cor:TrigonometricRegularization}
    Let $\w: [0,T) \to \mathbb{R}^N$ with $T \in \R_{>0} \cup \{+\infty\}$ be any solution to the differential equation \eqref{eq:gd_IRERM} in Definition \ref{definition:IRERM}, for $\A\in\mathbb{R}^{\M\times\N}$, $\y\in\mathbb{R}^{\M}$ with $\sigmaout \colon \R \to \R$ and $L \colon \R^\M \times \R^\M \to \R$
    satisfying \revFinal{Assumptions \ref{as:A1}, \ref{as:A3}, and \ref{as:B2}}.  Then, the following statements hold:
    
    \begin{enumerate}[label=(\roman*)]
        \item For $\sigmainn(\z) = \sinh(\z)$ and $\wprod_0 = \sigmainn(\w_0) = \0$, we have that if $\wprodT := \lim_{\t \to T} \sigmainn(\w(\t))$ with $\LossM(\wprodT) = 0$ exists, then
        \begin{equation}
             \g_\text{sinh} \left( \wprodT \right)
             =
              \min_{\widetilde\z\in\R^\N, \LossM(\widetilde\z) =0 } \g_\text{sinh} \left( \widetilde\z \right),
        \end{equation}
        where
        \begin{align}\label{eq:ImplicitRegularization_sinh}
            \g_\text{sinh} (\widetilde\z) = \langle \widetilde\z, \arctan(\widetilde\z) \rangle - \langle \bo, \tfrac{1}{2} \log(\bo +\widetilde\z^{\odot 2}) \rangle.
        \end{align}
        \label{cor:TrigonometricRegularizationNonAsymptotic_I}
        
        \item Assume in addition that \ref{as:B1} holds and that $\widetilde\z_0$ in \ref{as:B2} can be chosen such that $\widetilde\z_0 \in \R_{\neq 0}^\N$.\footnote{As already mentioned in Theorem \ref{thm:TrigonometricRegularizationIntroduction}, this property holds with probability one if $\A$ is generated at random with, e.g., iid Gaussian entries. \label{footnote:Special_z}} For $\sigmainn(\z) = \tanh(\z)$ and $\wprod_0 = \sigmainn(\w_0) = \0$, we have 
        that if $\wprodT := \lim_{\t \to T} \sigmainn(\w(\t))$ with $\LossM(\wprodT) = 0$ exists, then
        \begin{equation}
             \g_\text{tanh} \left( \wprodT \right)
             =
              \min_{\widetilde\z\in[-1,1]^\N, \LossM(\widetilde\z)=0} \g_\text{tanh} \left( \widetilde\z \right),
        \end{equation}
        where
        \begin{align}
            \g_\text{tanh} (\widetilde\z) = \langle \widetilde\z, \artanh(\widetilde\z) \rangle.
        \end{align}
        \label{cor:TrigonometricRegularizationNonAsymptotic_II}
    \end{enumerate}
\end{corollary}

\revFinal{\begin{remark}
The polynomial and trigonometric cases considered above are representative examples. More generally, similar conclusions can be derived for other choices of $\sigmainn$, provided that the associated function $h'$ satisfies the integrability conditions in Assumption~(A2). In this sense, the framework applies to a broader class of reparametrizations beyond the specific families considered here.
\end{remark}}

For the sake of conciseness, we defer the proof of Corollary \ref{cor:TrigonometricRegularization} to Section \ref{sec:Non-asymptoticBounds} and only provide a sketch of the argument here.

\begin{proof}[Proof sketch:]
    The proof idea of Corollary \ref{cor:TrigonometricRegularization} is very similar to the proof sketch of Corollary \ref{cor:PolynomialRegularization}.
    The main difference is that we may pick $\alpha = 0$ and have 
    \begin{align*}
        \langle \bo, \H(\widetilde\z) - \widetilde\z\odot\h(\wprod_0) \rangle = \g_\text{sinh}(\widetilde\z) - \arctan(\alpha) \langle \bo, \widetilde\z \rangle, 
    \end{align*}
    for $\sigmainn(\z) = \sinh(\z)$ with $\wDomainRP = \R^\N$, and
    \begin{align*}
        \langle \bo, \H(\widetilde\z) - \widetilde\z\odot\h(\wprod_0) \rangle =
        \frac{1}{2} \big( \g_\text{tanh}(\widetilde\z) - ( \artanh(\alpha) + \tfrac{\alpha}{1 - \alpha^2} ) \langle  \bo , \widetilde\z \rangle \big),
    \end{align*}
    for $\sigmainn(\z) = \tanh(\z)$ with  $\wDomainRP = [-1,1]^\N$, where $\H$ and $\h$ have been defined in Theorem \ref{theorem:vector_IRERM}. The claim follows by inserting $\alpha = 0$ into the equations above. 
\end{proof}

Just as in the polynomial case, we can use Theorem \ref{thm:Convergence} to analyze the convergence behaviour of $\wprod(\t)$ for $\sigmainn(\z) = \sinh(\z)$ and $\sigmainn(\z) = \tanh(\z)$.

\begin{corollary}[Trigonometric Regularization --- Convergence]
\label{cor:TrigonometricRegularizationConvergence}
    Let $\sigmainn(z) = \sinh(z)$ or $\sigmainn(z) = \tanh(z)$. Let $\w: [0,T) \to \mathbb{R}^N$ with $T \in \R_{>0} \cup \{+\infty\}$ be any solution to the differential equation \eqref{eq:gd_IRERM} in Definition \ref{definition:IRERM}, \revFinal{for $\A\in\mathbb{R}^{\M\times\N}$, $\y\in\mathbb{R}^{\M}$, $\sigmaout \colon \R \to \R$, and $L \colon \R^\M \times \R^\M \to \R$ satisfying Assumptions \ref{as:A1}, \ref{as:A3}, \ref{as:B2}, and \ref{as:B1}}. For $\sigmainn = \tanh$ furthermore assume that $\widetilde\z_0$ in \ref{as:B2} can be chosen such that $\widetilde\z_0 \in \R_{\neq 0}^\N$, \revFinal{cf.\ Footnote \ref{footnote:Special_z}}.
    Then the following statements hold:
    \begin{enumerate}[label=(\roman*)]
        \item The trajectory $\w$ can be extended to $\w \colon [0,\infty) \to \R^N$ and $\wprod(\t) = \sigmainn(\w(\t))$ converges to a limit $\wprodinfty \in \wDomainRP$ with $\LossM (\wprodinfty) = 0$ with the rate characterized in Theorem \ref{thm:Convergence} \ref{thm:Convergence_i}.
        Here, $\wDomainRP = \R^\N$ for $\sigmainn = \sinh$ and $\wDomainRP = (-1,1)^\N$ for $\sigmainn = \tanh$.
        \label{cor:TrigonometricRegularizationConvergenceRate_i}
        \item  If \revFinal{$\v \mapsto L(\sigmaout(\v),\y)$} satisfies the Polyak-Lojasiewicz inequality with $\mu > 0$, then there exists a constant $c > 0$, which is made explicit in the proof and only depends on $\mu$, $\widetilde\z_0$, $D_F(\widetilde\z_0,\wprod_0)$, and $\sigma_{\min}(\A)$, such that 
        \begin{align*}
            \LossM(\wprod(\t)) \le \LossM(\wprod_0) e^{-ct}.
        \end{align*}
        \label{cor:TrigonometricRegularizationConvergenceRate_ii}
    \end{enumerate}
\end{corollary}

\begin{proof}
    We already checked in the proof of Corollary \ref{cor:TrigonometricRegularization} that \ref{as:A2} holds.
    We furthermore verified that $\sigmainn' (x) \neq 0$ for all $x \in \R$ in both cases which implies that any flow is regular. All assumptions of Theorem \ref{thm:Convergence} are thus fulfilled. 
    Recall the definition of $h$ from Theorem \ref{theorem:vector_IRERM} and 
    let us split the cases:\\

    \noindent \textbf{Case $\sigmainn(z) = \sinh(z)$:} If $\sigmainn(z) = \sinh(z)$, we have that
    \begin{align*}
        h'(\widetilde z) = \big( [\sigmainn^{-1}]'(\widetilde z) \big)^2 = (1+\widetilde z^{2})^{-1},
    \end{align*}
    for $\widetilde z = \sigmainn(z)$, cf.\ proof of Corollary \ref{cor:TrigonometricRegularization}. We first check that the function $\widetilde z \mapsto \widetilde z \cdot h'(\widetilde z)$ lies in the function class $\mathcal{F}_\wDomainRPOneD$, where $\mathcal{F}_\wDomainRPOneD$ was defined in \eqref{eq:F_D}. By continuity $\widetilde z \mapsto \widetilde z \cdot h'(\widetilde z) \in L_1^c(\R^\N)$. Moreover, for any $\widetilde z \in \R$,
    \begin{align*}
        \Big| \int_{\widetilde z}^{\pm \infty} \frac{\tau}{1+\tau^2} d\tau \Big| = + \infty
    \end{align*}
    such that $\widetilde z \mapsto \widetilde z \cdot h'(\widetilde z) \in \mathcal{F}_\wDomainRPOneD$. 
    
        Now extend $\w$ to a maximal interval $[0,\tilde{T})$. 
        (Again, we slightly abuse notation and we do not relabel $\w$.)
        Note that this extension is well-defined since the solution of the gradient flow is unique which follows from the fact that $\sigmainn'$ is continuously differentiable and, thus, $\w'$ is a locally Lipschitz continuous function of $\w$, cf.\ the right-hand side of \eqref{equ:gradflowdefinition} below.
        Now we assume by contradiction that $\tilde{T} < + \infty$.
        In this case, Theorem \ref{thm:Convergence} \ref{thm:Convergence_v} yields
        that $\underset{t \in [0, \tilde{T} )}{\sup} \| \wprod(\t) \|_2  < + \infty$.
        By applying Theorem \ref{thm:Convergence} \ref{thm:Convergence_iii} we can extend $\w$ to a gradient flow solution $\w: [0,\tilde{T}+\varepsilon) \rightarrow \R$ for some $\varepsilon>0$.
        This is a contradiction to the assumption that $\tilde{T} < + \infty $ was chosen maximal.
        Thus, $\w$ can be extended to a function $\w: [0,+\infty) \rightarrow \R$.
        The remaining part of Claim \ref{cor:TrigonometricRegularizationConvergenceRate_i} follows directly from Theorem \ref{thm:Convergence} \ref{thm:Convergence_i} and \ref{thm:Convergence_iv}.

    To see why Claim \ref{cor:TrigonometricRegularizationConvergenceRate_ii} holds, note that $h'(\wprod(\t))^{\odot-1} = \bo+\wprod(\t)^{\odot 2} \ge \bo$.
    If we assume in addition that \revFinal{$\v \mapsto L(\sigmaout(\v),\y)$} satisfies the Polyak-Lojasiewicz inequality with $\mu > 0$, the claim follows from Theorem \ref{thm:Convergence} \ref{thm:Convergence_ii} with $c = 2\mu \sigma_{\min}(\A)^2$, \revFinal{cf.\ Remark \ref{rem:PL} (ii), which can be applied for $[\t_0,\t_1) = [0,T)$ \revFinal{since $T_\circ = \emptyset$}}.\\

    \noindent \textbf{Case $\sigmainn(z) = \tanh(z)$:} If $\sigmainn(z) = \tanh(z)$, one can compute that
    \begin{align*}
        h'(\widetilde z) = \big( [\sigmainn^{-1}]'(\widetilde z) \big)^2 = (1-\widetilde z^2)^{-2},
    \end{align*}
    for $\widetilde z = \sigmainn(z)$.    
    We can extend $\w$ to a maximal interval $[0,\tilde{T})$.
    (Note that this extension is unique and well-defined since $\tanh'$ is locally Lipschitz continuous, which implies that $\w'$ is a locally Lipschitz continuous function of $\w$, cf.\ the right-hand side of \eqref{equ:gradflowdefinition} below.)
    Again, we assume by contradiction that $\tilde{T}< + \infty$.
    Since $\wDomainRP = (-1,1)^\N$ is bounded, 
    Theorem \ref{thm:Convergence} \ref{thm:Convergence_iii} yields that we must have $\lim_{\t \to T} \wprod(\t) \in \partial \wDomainRP$ 
    since $\w$ cannot be extended to a larger interval $[0,T+\varepsilon)$ by assuming that $\tilde{T}$ is maximal. 
    However, this contradicts the fact that $|\wprode_i(t)| \le \gamma <1$, for any $i \in \N$ and $t \in [0,\tilde T)$, where $\gamma$ only depends on $\widetilde\z_0$ and $D_F(\widetilde\z_0, \wprod_0)$. Indeed, by Lemma \ref{lem:decay} and the shape of $\H(\widetilde z) = \frac{1}{2} \widetilde z \ \artanh (\widetilde z)$, which we computed in the proof of Corollary \ref{cor:TrigonometricRegularization}, we easily get that
    \begin{align*}
       D_F(\widetilde\z_0,\wprod_0)
       &\ge D_F(\widetilde\z_0,\wprod(t))
       \ge D_H((\widetilde z_0)_i, \wprode_i(t)) \\
       &= \frac{1}{2} \left( (\widetilde z_0)_i (\artanh((\widetilde z_0)_i) - \artanh(\wprode_i(t)) + \frac{(\widetilde z_0)_i \wprode_i(t)}{1-(\widetilde z_0)_i^2} - \frac{(\widetilde z_0)_i^2}{1-(\widetilde z_0)_i^2} \right),
    \end{align*}
    for any $i \in [\N]$. Together with the fact that $|\wprod(t)|\le 1$ and our additional assumption that $\widetilde \z_0 \in \R_{\neq 0}^\N$, this implies that
    \begin{align}
    \label{eq:w_lowerbound}
        \wprode_i(t) \ge \tanh \left( \artanh((\widetilde z_0)_i) - \frac{1}{(\widetilde z_0)_i} \left( 2 D_F(\widetilde\z_0,\wprod_0) + \frac{(\widetilde z_0)_i + (\widetilde z_0)_i^2}{1-(\widetilde z_0)_i^2} \right) \right)
        =: \gamma_-
        > -1.
    \end{align}
    Similarly, we get that
    \begin{align*}
        0 &\le D_F(\widetilde\z_0,\wprod(t))
        = D_H((\widetilde z_0)_i, \wprode_i(t)) + \sum_{j \neq i} D_H((\widetilde z_0)_j, \wprode_j(t)) \\
        &\le D_H((\widetilde z_0)_i, \wprode_i(t)) + D_F(\widetilde\z_0,\wprod_0),
    \end{align*}
    for any $i \in [\N]$, implying that
    \begin{align}
    \label{eq:w_upperbound}
        \wprode_i(t) \le \tanh \left( \artanh((\widetilde z_0)_i) + \frac{1}{(\widetilde z_0)_i} \left( 2 D_F(\widetilde\z_0,\wprod_0) + \frac{(\widetilde z_0)_i - (\widetilde z_0)_i^2}{1-(\widetilde z_0)_i^2} \right) \right)
        =: \gamma_+
        < 1
    \end{align}    
    and thus $|\wprode_i(t)| \le \gamma := \max\{|\gamma_-|,|\gamma_+|\} <1$, for any $i \in \N$ and $t \in [0,\tilde T)$.    
    Since this yields $\tilde{T}=\infty$ by contradiction,
    Theorem \ref{thm:Convergence} \ref{thm:Convergence_i} and \ref{thm:Convergence_iv} conclude the proof of Claim \ref{cor:TrigonometricRegularizationConvergenceRate_i}.

    To deduce Claim \ref{cor:TrigonometricRegularizationConvergenceRate_ii} from Theorem \ref{thm:Convergence}  \ref{thm:Convergence_ii}, note that $h'(\wprod(\t))^{\odot-1} = (\bo-\wprod(\t)^{\odot 2})^{\odot 2} \ge (1-\gamma^2)^2$ such that we can set $c = 2(1-\gamma^2)^2 \mu \sigma_{\min}(\A)^2$, \revFinal{cf.\ Remark \ref{rem:PL} (ii)}.
\end{proof}

\section{Proofs}
\label{sec:Proof}

In this section, we provide the proofs of our main results.

\subsection{Proof of Theorem \ref{theorem:vector_IRERM}}
\label{sec:Proofvector_IRERM}

The proof of Theorem \ref{theorem:vector_IRERM} is based on the following technical observation. 

\begin{lemma}\label{theorem:vector_general}
Let $\A\in\mathbb{R}^{\M\times\N}$, $\boldsymbol{\upsilon}\in\mathbb{R}^{\M}$, 
and let $\wDomainRPOneD \subset \R$ be an open, convex interval.
Suppose that $\wprod: [0,T) \rightarrow \wDomainRP$, for $T \in \R \cup \{+\infty\}$, is continuously differentiable on $(0,T)$ such that
\begin{enumerate}
    \item\label{assumption_converge} 
    the limit $\wprodT:=\lim_{\t\to T}\wprod(\t)  \in \R^N$ exists, %
    \item\label{assumption_linear} $\A^{T}\A \wprodT =\A^{T}\boldsymbol{\upsilon}$.
\end{enumerate}
Furthermore, assume that there exist continuous functions $\g: \wDomainRP \to \mathbb{R}^{\N}$ and $\h:\wDomainRPOneD \to \mathbb{R}$, \revFinal{and $T_\circ \subset [0,T)$ with $|T_\circ| < \infty$} such that the following properties hold:
\begin{enumerate}
    \setcounter{enumi}{2}
    \item\label{assumption_differentiable} $h \circ \wprod$ is continuously differentiable on any open interval in $[0,T) \setminus T_\circ$,
    \item\label{assumption_identity} for all $t \in [0,T) \setminus T_\circ$, 
    \begin{equation}\label{eq:ode_general}
        \h'(\wprod(\t)) \odot \wprod'(\t)
        = \g(\wprod(\t)),
    \end{equation}

    \item\label{assumption_derivative} \revFinal{$h'$ is well-defined almost everywhere with $\h' \in L_1^c(\wDomainRPOneD)$ and $h' > 0$ almost everywhere,}
    \item\label{assumption_kernel} $g\left( \widetilde\z \right) \in \ker (\A)^{\bot} $ for all $\widetilde\z \in \wDomainRP$.
\end{enumerate}
Let $\H: \wDomainRPOneD \to \R$ be an antiderivative of $h$, i.e., $\H' = \h$.
Then both $H$ and $h$ can be extended to continuous functions $\bar{H}: \overline{\wDomainRPOneD} \to \R \cup \left\{ +\infty \right\} $ and $ \bar{h}: \overline{\wDomainRPOneD} \to \R\cup \left\{ -\infty; +\infty \right\}$.
\footnote{On $ \R \cup \left\{ -\infty; +\infty \right\}$ we consider the closure with respect to the usual topology that is homeomorphic to $ [0,1] $.
By $ \overline{\wDomainRPOneD}$ we denote the (topological) closure of 
$\wDomainRPOneD$ in $\R$.
}
If any of the following two conditions is satisfied 
\begin{enumerate}
    \item[(i)] $|  \bar{h}  (x)| < \infty$ and $\bar{\H} (x) < \infty$ for all $x \in \overline{\wDomainRPOneD}$,
    \item[(ii)] $\wprodT \in \wDomainRP $,
\end{enumerate}
then 
\begin{equation}\label{eq:optimal_vector_general}
    \wprodT \in \argmin_{\zprod \in\overline{\wDomainRP}: \A^{T}\A \zprod=\A^{T}\boldsymbol{\upsilon}} \langle \bo, \bar{H} ( \zprod ) - \zprod \odot \bar{h} (\wprod(0)) \rangle.
\end{equation}
\end{lemma}
\begin{proof}
	We abbreviate $\wprod_0 = \wprod(0)$.
    First, note that $\H: \wDomainRPOneD \to \R$ is convex due to Assumption \ref{assumption_derivative}.
    This implies that it can be extended continuously to $\bar{H}: \overline{\wDomainRPOneD} \to \R \cup \left\{ +\infty \right\} $.
    Again using Assumption \ref{assumption_derivative} we observe that the function $\h: \wDomainRPOneD \to \R$ is monotonically increasing.
    This implies that it can be extended to a continuous function $ \bar{h}: \overline{\wDomainRPOneD} \to \R\cup \left\{ -\infty; +\infty \right\}$.

    The proof is based on validating the KKT-conditions of \eqref{eq:optimal_vector_general} for the limit 
    $\wprodT$.
    Let us first consider the case that (i) holds, i.e., that $|\h (x)| < \infty$ and $\H (x)< \infty$  for all $x \in \overline{\wDomainRPOneD}$.    
    \revFinal{Let $t' \in (0,T)$ be arbitrary.}
    Recalling Assumption \ref{assumption_identity} and writing \revFinal{$\T_\circ \cap [0,t'] = \{ t_1, \dots, t_{m-1} \}$} adding $t_0 = 0$ and $t_m := t'$, we have that 
    \begin{align}
        \int_{0}^{ t'} \g(\wprod(\t))d\t
        &= 
        \lim_{\varepsilon \searrow 0} \sum_{k = 1}^m \int_{t_{k-1}+\varepsilon}^{t_k-\varepsilon}
        \g(\wprod(\t))
        d\t
        \nonumber \\
        &= \lim_{\varepsilon \searrow 0} \sum_{k = 1}^m \int_{t_{k-1}+\varepsilon}^{t_k-\varepsilon}\wprod'(\t)\odot \h'(\wprod(\t)) d\t
        \nonumber \\
        &=
        \lim_{\varepsilon \searrow 0} \sum_{k = 1}^m 
        \left(
         \h(\wprod \left( t_k -\varepsilon \right)) - 
         \h( \wprod (  t_{k-1} +\varepsilon ) )
        \right)
        = \h(\wprod \left( t' \right)) - \h(\wprod_0),
           \label{eq:IntegrationOfG}
    \end{align}
    where we use that $h \circ \wprod$ is continuous on $[0,T)$ and continuously differentiable on open intervals in $[0,T) \setminus T_\circ$ by Assumption \ref{assumption_differentiable}.
    Since $\wprod(t) \in \wDomainRP$, for $t \in (0,T)$, we obtain by Assumption \ref{assumption_kernel} for any $t' \in [0,T)$ that
    \begin{equation*}
         \h(\wprod \left(t'\right)) - \h(\wprod_0)  \in \ker ( \A )^\bot.
    \end{equation*}
    Using Assumption \ref{assumption_converge} together with continuity of $h$, we can take the limit $t' \rightarrow T$ to obtain that 
    \begin{equation*}
         \h( \wprodT ) - \h(\wprod_0)  \in \ker ( \A )^\bot.
    \end{equation*}
    Together with Assumption \ref{assumption_linear}, we hence have that $\wprodT$ satisfies the following conditions
    \begin{align}
        \wprodT &\in \overline{\wDomainRP},\label{eq:xcondition_1}\\
        \0 &= \A^{T}\A \wprodT -\A^{T}\boldsymbol{\upsilon}, \label{eq:xcondition_2}\\
         \h(\wprodT) - \h(\wprod_0)  &\in \ker (\A)^\bot. %
        \label{eq:xcondition_3}
    \end{align}
    
    We now examine \eqref{eq:optimal_vector_general} and derive its KKT conditions.
    Since $\overline{\wDomainRP}$ is a convex polytope, there exist $\B \in \R^{\J\times \N}$, $\d \in \R^\J$, for $\J \in \mathbb N$, such that $\overline{\wDomainRP} = \{\zprod\in\mathbb{R}^{\N}: \B \zprod - \d \leq \0 \}$. 
    Consequently, the Lagrangian of \eqref{eq:optimal_vector_general} is given by
    \begin{align*}
        \mathfrak{L}(\zprod,\bm\lambda,\bm\mu)
        = \langle \bo, \H(\zprod) - \zprod\odot\h(\wprod_0) \rangle
        + \langle \bm{\lambda},\A^{T}\A\zprod - \A^{T}\boldsymbol{\upsilon} \rangle
        + \langle \bm\mu, \B\zprod - \d \rangle,
    \end{align*}
    where $\bm\lambda \in \R^{N}$ and $\bm\mu \in \R^\J$.
    The KKT conditions are given by
    \begin{align*}
        \0 &= \h(\zprod) - \h(\wprod_0) + \A^{T}\A \bm\lambda + \B^T \bm{\mu}
        && \text{(stationarity)}\\
        \0 &= \A^{T}\A\zprod - \A^{T}\boldsymbol{\upsilon} 
        &&\text{(primal feasibility)}\\
        \0 &\geq \B\zprod - \d \\
        \0 &\leq \bm{\mu}
        &&\text{(dual feasibility)}\\
        0 &= \langle \bm{\mu}, \B \zprod - \d \rangle
        &&\text{(complementary slackness)}.
    \end{align*}
    Since the constraints of the optimization problem in \eqref{eq:optimal_vector_general} are affine functions, the KKT-conditions are necessarily fulfilled by solutions of the problem. 
    Furthermore, since the objective in \eqref{eq:optimal_vector_general} is convex (due to Assumption \ref{assumption_derivative}), any $(\widetilde\z,\bm\lambda,\bm\mu)$ satisfying the KKT-conditions provides a solution $\widetilde\z$ of \eqref{eq:optimal_vector_general}, see, e.g., \cite[Theorem 28.3]{rockafellar1970convex}. 
    Hence, all that remains is to show that there exist $\bm\lambda$ and $\bm\mu$ such that $(\wprodinfty,\bm\lambda,\bm\mu)$ satisfies the KKT-conditions.
    
    By conditions \eqref{eq:xcondition_1} and \eqref{eq:xcondition_2}, we see that $\wprodinfty$ obeys the primal feasibility. 
    Moreover, by choosing $\bm\mu = \0$, we automatically satisfy dual feasibility and complementary slackness.
    Now note that, in this case, stationarity reduces to 
    \begin{align*}
        \0 &= \h( \wprodT ) - \h(\wprod_0) +\A^{T}\A \bm\lambda.
    \end{align*}
    Since $\h(\wprodinfty) - \h(\wprod_0) \in \ker \left( \A \right)^\bot = \mathrm{range}(\A^T)$ by \eqref{eq:xcondition_3} we can find a $\bm \lambda$ such that the above equation holds. Just note that $\A^T$ is invertible if restricted to $\mathrm{range}(\A)$.

    Let us now consider the case that (i) does not hold, which implies that $\wprodT \in \wDomainRP$ must hold by (ii).
    In this setting, we may choose
    an open and convex interval $\wDomainOneD_\circ \subsetneq \wDomainRPOneD$ 
    \revFinal{
    such that $ \wprod (t) \in \wDomainCirc $
    for all $t \in (0,T)$
    }
    and such that both $ \bar{h} \vert_{\overline{\wDomainCircOneD}} $ and $ \bar{H} \vert_{ \overline{\wDomainCircOneD}} $ take only values in $\R$.
    Then, Assumptions \ref{assumption_converge}-\ref{assumption_kernel} hold if  $ \wDomainRPOneD $ is replaced by $ \wDomainOneD_\circ $ and our previous KKT argument shows that
    \begin{equation}\label{intern:lastone}
        \wprodT \in \argmin_{\zprod \in\overline{\wDomainCirc}: \A^{T}\A \zprod=\A^{T} \boldsymbol{\upsilon}} 
        \underset{=:F(\zprod)}{\underbrace{\langle \bo, \bar{H} ( \zprod ) - \zprod \odot \bar{h} (\wprod(0)) \rangle}},
    \end{equation}
    such that the claim \eqref{eq:optimal_vector_general} follows from the convexity of the function $F $.
    Indeed, assume by contradiction that there were a $\zprod \in \wDomainRP \setminus \overline{\wDomainCircOneD} $
    such that $ F(\zprod) < F(\wprodT)$ and $\A^{T}\A \zprod=\A^{T} \boldsymbol{\upsilon}$.
    Then, by choosing $ \lambda \in (0,1)$ sufficiently small we had that 
    $(1-\lambda) \wprodT + \lambda \zprod \in \wDomainCirc $. 
    However, since
    \begin{equation*}
        F \left( (1-\lambda) \wprodT + \lambda \zprod \right)
        \le 
        (1-\lambda) F (\wprodT) + \lambda F(\zprod)
        < F (\wprodT)
    \end{equation*}
    and since $  \A^{T}\A \left( (1-\lambda) \wprodT + \lambda \zprod \right) =\A^{T} \boldsymbol{\upsilon}$,
    this would contradict \eqref{intern:lastone}.
    
\end{proof}

With Lemma \ref{theorem:vector_general} at hand, we can now prove Theorem \ref{theorem:vector_IRERM}.

\begin{proof}[Proof of Theorem \ref{theorem:vector_IRERM}]
By definition, the derivative of $\w$ on $(0,T)$ is
\begin{align}\label{equ:gradflowdefinition}
    \w'(\t)
    &= - \nabla\mathcal{L}(\w(t))
    = - \Big[ \A^T \nabla_\v L\big( \sigmaout(\v), \y \big) \big|_{\v = \A\sigmainn(\w(t))} \Big] \odot \sigmainn'(\w(t)),
\end{align}
where we used the chain rule in the second equality. Consequently,
\begin{align}    
    \wprod'(\t)
    &= \sigmainn'(\w(t)) \odot \w'(\t)
    = \sigmainn'(\w(t)) \odot [-\nabla\mathcal{L}(\w(t))]  \nonumber \\
    &= [\sigmainn'(\w(t))^{\odot 2}] \odot \Big[- \A^T \cdot \nabla_\v L\big( \sigmaout(\v), \y \big) \big|_{\v = \A\wprod(t)} \Big].  \label{eq:with_sigmaout_PreVersion}
\end{align}
By definition, we know that the trajectory $\wprod (\cdot)$ stays in $\wDomainRP$ such that we can write

\begin{align}
   \label{eq:with_sigmaout}
   \begin{split}
       \wprod'(\t)
       = 
       [\sigmainn'(\w (t))^{\odot 2}] 
       \odot \underbrace{\Big[- \A^T \cdot \nabla_\v L\big( \sigmaout(\v), \y \big) \big|_{\v = \A\wprod(t)} \Big]}_{=:\g(\wprod(t))},
   \end{split}
\end{align}
where $g \colon \wDomainRP \to \R^\N$ is continuous by Assumptions \ref{as:A1} and \ref{as:A3}. 
It is important to note that if $\sigmainn'( \w(t) )_i \neq 0$ for some $t \in (0,T)$ and some $i \in [N] $, then
\begin{align}
\label{eq:h_prime_explicit}
    [\h'(\wprod(t))]_i = \big[ [\sigmainn^{-1}]'(\wprod(t)) \big]_i^2 = \frac{1}{[\sigmainn'(\w(\t))]_i^2}
\end{align}
is well-defined, where $h'$ has been defined in \eqref{eq:hprimedefinitionWithoutPositivity}.  
By assuming $\w$ to be regular, we know that there are at most finitely many $\t \in (0,T)$ such that $[\sigmainn'(\w(\t))]_i = 0$ for some $i \in [\N]$.
Together with \eqref{eq:with_sigmaout} this implies that the equation 
\begin{align}
\label{eq:identity}
    h'(\wprod (t)) \odot \wprod' (t) = g(\wprod (t))
\end{align}
holds for all but finitely many $\t \in (0,T)$,
as required in Assumption \ref{assumption_identity} of Lemma \ref{theorem:vector_general}. Furthermore note that $g$ is continuous by Assumptions \ref{as:A1} and \ref{as:A3}, that $h$ is continuous by Assumption \ref{as:A2}, cf.\ Footnote \ref{footnote:h}, and that $\wprod$ is continuously differentiable on $(0,T)$ by \eqref{eq:with_sigmaout} and \ref{as:A1}-\ref{as:A3}. Denoting the set on which \eqref{eq:identity} holds by $[0,T) \setminus T_0$, we deduce from $g \circ \wprod$ being continuous and \eqref{eq:identity} that $h \circ \wprod$ is continuously differentiable on any open interval in $[0,T) \setminus T_0$, as required in Assumption \ref{assumption_differentiable} of Lemma \ref{theorem:vector_general}.

Recall that we assume $\wT:=\lim_{\t\to T}\w(\t)$ with $\mathcal{L}(\wT) = 0$ exists and that (i) $\wprodT := \sigmainn( \wT ) \in \wDomainRP$ or (ii) the continuous extensions of $h$ and $H$ to $\overline{\wDomainRPOneD}$ are finite valued.
Set $\boldsymbol{\upsilon} = \sigmaout^{-1}(\y)$, which is well-defined by Assumption \ref{as:A1}. We wish to apply Lemma \ref{theorem:vector_general} to $\wprod = \sigmainn(\w)$. We already verified that $\g: \wDomainRP \to \mathbb{R}^{\N}$ and $\h:\wDomainRPOneD \to \mathbb{R}$ are continuous functions, that $\wprod$ is continuously differentiable on $(0,T)$, and that Assumptions \ref{assumption_differentiable} and \ref{assumption_identity} hold. Let us now check the remaining Assumptions \ref{assumption_converge}-\ref{assumption_linear} and \ref{assumption_derivative}-\ref{assumption_kernel} of Lemma \ref{theorem:vector_general}.

Assumption \ref{assumption_converge} of Lemma \ref{theorem:vector_general} is clearly satisfied.
To validate Assumption \ref{assumption_linear}, recall \ref{as:A1} and \ref{as:A3} which yield that $\mathcal L (\w_T) = 0$ if and only if $\sigmaout(\A \wprodT) = \y$ which in turn is equivalent to $\A \wprodT = \boldsymbol{\upsilon}$ by injectivity of $\sigmaout$.
By \eqref{eq:hprimedefinitionWithoutPositivity} and Assumption \ref{as:A2}, the function $h'$ is well-defined in all but finitely many points and is strictly positive whenever it exists which yields Assumption \ref{assumption_derivative}.
Finally, Assumption \ref{assumption_kernel} holds 
since $\g(\wprod) \in \mathrm{range}(\A^T) = \ker(\A)^\perp$.

Applying Lemma \ref{theorem:vector_general} and using that $\{ \widetilde\z \in \R^ N \colon \A\widetilde\z = \boldsymbol{\upsilon} \} \subset \{ \widetilde\z \in \R^ N \colon \A^T\A\widetilde\z = \A^T\boldsymbol{\upsilon} \}$, we obtain
\begin{align*}
    \wprodT
    &\in \argmin_{\widetilde\z\in\overline{\wDomainRP}, \A\widetilde\z = \boldsymbol{\upsilon}
    } \langle \bo, \H(\widetilde\z) - \widetilde\z\odot\h(\wprod_0) \rangle\\
    &= \argmin_{\widetilde\z\in\overline{\wDomainRP},\,\LossM(\widetilde\z) = 0} \langle \bo, \H(\widetilde\z) - \widetilde\z\odot\h(\wprod_0) \rangle,
\end{align*} 
where we recalled from above that $\LossM(\widetilde\z) = 0$ if and only if $\A \widetilde\z = \boldsymbol{\upsilon}$. This yields \eqref{eq:implicit_regularization_Bregman}.
To conclude, we only need to subtract from the right-hand side the quantity $(\langle \bo, \H(\wprod_0) \rangle - \langle \wprod_0, \h(\wprod_0) \rangle)$, which is independent of $\widetilde\z$. For $F(\widetilde\z) = \langle \bo, \H(\widetilde\z)\rangle$, we then have
\begin{align}
\label{eq:MinEquivalence}
\begin{split}
    &\argmin_{\widetilde\z\in\overline{\wDomainRP},\,\LossM(\widetilde\z) = 0} \langle \bo, \H(\widetilde\z) - \widetilde\z\odot\h(\wprod_0) \rangle \\
    &= \argmin_{\widetilde\z\in\overline{\wDomainRP},\,\LossM(\widetilde\z) = 0} \langle \bo, \H(\widetilde\z) - \widetilde\z\odot\h(\wprod_0) \rangle -  \langle \bo, \H(\wprod_0) \rangle + \langle \wprod_0, \h(\wprod_0) \rangle\\
    &= \argmin_{\widetilde\z\in\overline{\wDomainRP},\,\LossM(\widetilde\z) = 0} \langle \bo, \H(\widetilde\z)-\H(\wprod_0) \rangle - \langle \h(\wprod_0), \widetilde\z - \wprod_0  \rangle\\
    &= \argmin_{\widetilde\z\in\overline{\wDomainRP},\,\LossM(\widetilde\z) = 0} D_{F}(\widetilde\z,\wprod_0).
\end{split}
\end{align}

\end{proof}

\subsection{Proof of Theorem \ref{thm:Convergence}}
\label{sec:ProofConvergence}

Let us begin by collecting some useful observations:
\begin{itemize}
    \item As argued in the proof of Theorem \ref{theorem:vector_IRERM}, $h' \circ \wprod$ is well-defined and continuously differentiable on any open intervall in $[0,T) \setminus T_\circ$, where $T_\circ \subset \R$ with $|T_\circ| < \infty$, by regularity of $\w$. In particular, 
    \begin{align}
    \label{eq:identity2}
        h'(\wprod (t)) \odot \wprod' (t) = g(\wprod (t))
    \end{align}
    holds for any $t \in [0,T) \setminus T_\circ$, \revFinal{with $g$ defined as in \eqref{eq:with_sigmaout}.}
    
    \item By Assumption \ref{as:B2}, there exists $\widetilde\z_0 \in \wDomainRP$ with $\revFinal{\sigmaout(\A\widetilde\z_0)} = \y$ such that Assumption \ref{as:A3} implies $\LossM(\widetilde\z_0) = 0$. 
    \item Let $D_F$ be the Bregman divergence associated with $F(\z) = \langle \bo, H(\z) \rangle$, where $H$ is an antiderivative of $h$.    
    (The Bregman divergence is invariant under affine transformation and, thus, it does not matter which antiderivatives we choose, see also Footnote \ref{footnote:H}).     
    Then, the function $\t \mapsto D_{F}(\widetilde\z_0,\wprod(\t))$ is decreasing and $D_{F}(\widetilde\z_0,\wprod(\t))$ converges for 
    \revFinal{$\t \to T$}. 
\end{itemize}

The last point is formalized in the lemma below which can be seen as a generalization of \cite[Lemma 8]{chou2021more}.

\begin{lemma}
\label{lem:decay}
    Let $\w: [0,T) \to \mathbb{R}^N$ with $T \in \R_{>0} \cup \{+\infty\}$ be any 
    regular
    solution to the differential equation  \eqref{eq:gd_IRERM} in Definition \ref{definition:IRERM}, for $\A\in\mathbb{R}^{\M\times\N}$, $\y\in\mathbb{R}^{\M}$, $\sigmainn \colon \R \to \R, \ \sigmaout \colon \R \to \R$, and $L \colon \R^\M \times \R^\M \to \R$. \revFinal{Assume that \ref{as:A1}, \ref{as:A2}, \ref{as:A3} and \ref{as:B1}} hold, and let $F(\z) = \langle \bo, H(\z) \rangle$, where $H$ is an antiderivative of $h$ as defined in Theorem \ref{theorem:vector_IRERM}. Then, for any $\widetilde\z \in \wDomainRP$  %
    and every $t \in (0,T) \setminus T_\circ$,
    \begin{align*}
         \frac{d}{dt} D_{F}(\widetilde\z,\wprod(\t)) 
        \le \LossM(\widetilde\z) - \LossM(\wprod(\t)).
    \end{align*}
    Here, $T_\circ$ is the finite set of times at which $h' \circ \wprod$ is not well-defined, cf.\ discussion around \eqref{eq:identity2}.
    In particular, if $T = \infty$ and $\LossM(\widetilde\z) = 0$,
    then $\lim_{t \rightarrow + \infty} D_\F(\widetilde\z,\wprod(\t))$ exists.
\end{lemma}

\begin{proof}
    Let $\t \in (0,T) \setminus T_\circ$. As discussed below \eqref{eq:with_sigmaout}, the function $h' \circ \wprod$ is well-defined and continuously differentiable on a neighbourhood of $\t$, and \eqref{eq:identity2} holds.
	Applying now the chain rule, we compute the time derivative of $\t \mapsto D_{F}(\widetilde\z,\wprod(\t))$ as
	\begin{align}
		\frac{d}{dt} D_{F}(\widetilde\z,\wprod(\t))
		&= - \langle \nabla F(\wprod(\t)),\wprod'(\t) \rangle - \langle \partial_\t \nabla F(\wprod(\t)), \widetilde\z - \wprod(\t) \rangle + \langle \nabla F(\wprod(\t)),\wprod'(\t) \rangle \nonumber\\
		&= - \langle \h'(\wprod(\t)) \odot\wprod'(\t), \widetilde\z-\wprod(\t) \rangle
		= \big\langle \A^T \cdot \nabla_\v L\big( \revFinal{\sigmaout(\v)}, \y \big) \big|_{\v = \A\wprod(\t)}, \widetilde\z-\wprod(\t) \big\rangle  \nonumber\\
		&= \big\langle \nabla_\v L\big( \revFinal{\sigmaout(\v)}, \y \big) \big|_{\v = \A\wprod(\t)}, \A\widetilde\z - \A\wprod(\t)  \big\rangle \nonumber\\
		&\le L\big( \revFinal{\sigmaout(\A\widetilde\z)}, \y \big) - L\big( \revFinal{\sigmaout(\A\wprod(\t))}, \y \big) \nonumber \\
		&= \LossM(\widetilde\z) - \LossM(\wprod(\t)),\label{eq:decay}
	\end{align}
	where we used in the second line \eqref{eq:identity2}, and in the fourth line the convexity of \revFinal{$\v \mapsto L(\sigmaout(\v),\y)$ according to \ref{as:B1}}.
	We further know that $D_F(\widetilde \z,\wprod(0)) < + \infty$ due to $H$ being the anti-derivative of a continuous function. 
    For $T = \infty$ and $\LossM(\widetilde\z) = 0$, this immediately implies that $D_\F(\widetilde\z,\wprod(\t))$ converges for $\t \to + \infty$ since $\frac{d}{dt} D_{\F}(\widetilde\z,\wprod(\t))\leq 0$, for all $\t \in (0,T) \setminus T_\circ$, by \eqref{eq:decay} and Assumption \ref{as:A3}, and $D_\F(\widetilde\z,\wprod(\t)) \geq 0$ by strict convexity of $F$. The latter follows from $\nabla^2 F(\z) = \mathrm{diag}(h'(\z))$ and the fact that $h'(z)$ is strictly positive for all but finitely many $z \in \R$ by \eqref{eq:hprimedefinitionWithoutPositivity} and Assumption \ref{as:A2}.
\end{proof}

We can now proceed with proving Claims \ref{thm:Convergence_i}-\ref{thm:Convergence_v} of Theorem \ref{thm:Convergence}. Let us abbreviate $\wprod_0 = \wprod(0)$.\\  

\emph{Proof of \ref{thm:Convergence_i}:} 
We use Lemma \ref{lem:decay}
to show that \ref{thm:Convergence_i} holds. Let $T_\circ \subset \R$ be the finite set defined in Lemma \ref{lem:decay}. Recall that $D_{\F}$ is non-negative, that $\t \mapsto D_{F}(\widetilde\z_0,\wprod(\t))$ is continuous on $[0,T)$ and continuously differentiable on any open interval in $[0,T) \setminus T_\circ$, and that $D_F(\widetilde \z_0,\wprod_0) < \infty$ (as argued in the proof of Lemma \ref{lem:decay}). We thus get for any $t \in (0,T) \setminus T_\circ$ that
\begin{align*}
	D_\F(\widetilde \z_0, \wprod_0)
	&\geq D_{\F}(\widetilde \z_0,\wprod(0)) - D_{\F}(\widetilde \z_0,\wprod(t))
	= -\int_{0}^{t} \partial_{\t} D_{\F}(\widetilde \z_0,\wprod(\tau))d\tau\\
	&\ge \int_{0}^{t} \LossM(\wprod(\tau))d\tau
	\geq t \LossM(\wprod(t)),
\end{align*}
where the last inequality comes from the \revFinal{facts that $\LossM(\widetilde \z_0) = 0$ by assumption and that} $\t \mapsto \LossM(\wprod(\t))$ is non-increasing by the definition of $\w$ in \eqref{eq:gd_IRERM}. 
Rearranging the inequality yields the claim.
If $T = \infty$, we furthermore see that $\lim_{\t \to \infty} \LossM(\wprod(\t)) = 0$ since $D_F(\widetilde \z_0,\wprod_0) < + \infty$. 

\emph{Proof of \ref{thm:Convergence_ii}:} \revFinal{Recall that $\t_1,\t_2 \in [0,T]$ such that $[\t_1,\t_2] \cap T_\circ = \emptyset$.} The just observed sublinear convergence rate can be improved if we assume in addition that $(h')^{\odot -1}(\wprod(\t)) \ge r$, for some $r > 0$ and \revFinal{every $\t \in [\t_0,\t_1]$},\footnote{In contrast to $h' \circ \wprod$, which is only defined on $[0,T) \setminus T_\circ$, one can deduce from \eqref{eq:h_prime_explicit} and \ref{as:A2} that $(h')^{\odot -1} \circ \wprod$ is well-defined for all \revFinal{$\t \in [\t_0,\t_1]$}.} and that \revFinal{$\v \to L(\sigmaout(\v),\y)$ satisfies the Polyak-Lojasiewicz inequality with $\mu > 0$}. 
In this case, \revFinal{Equation \eqref{eq:identity2}} yields
\begin{align*}
	\partial_\t \LossM(\wprod(\t))
	&= \partial_\t \L(\revFinal{\sigmaout(\A\wprod(\t))},\y)
	= \Big\langle \nabla_{\wprod} L\big( \revFinal{\sigmaout(\A\wprod(\t))}, \y \big),\wprod'(\t) \Big\rangle \\
	&= -\Big\langle \nabla_{\wprod} L\big( \revFinal{\sigmaout(\A\wprod(\t))}, \y \big), h'(\wprod(\t))^{\odot -1} \odot \nabla_{\wprod} L\big( \revFinal{\sigmaout(\A\wprod(\t))}, \y \big) \Big\rangle \\
	&\le -r \Big\| \nabla_{\wprod} L\big( \revFinal{\sigmaout(\A\wprod(\t))}, \y \big) \Big\|_2^2 \\
	&\le - 2\revFinal{\mu} r \big( L(\revFinal{\sigmaout(\A\wprod(\t))},\y) - L(\revFinal{\sigmaout(\A\widetilde \z_0)},\y) \big) \\
	&= -2\revFinal{\mu} r \LossM(\wprod(\t)),
\end{align*}
for every $\t \in (0,T)$, where the penultimate step follows from the Polyak-Lojasiewicz inequality, 
and the last equality uses that $\revFinal{\sigmaout(\A\widetilde\z_0)} = \y$. 
Applying Grönwall's inequality on the interval $[\t_0,\t) \subset [\t_0,\t_1]$, we obtain that
\begin{align*}
	\LossM(\wprod(\t)) \le \revFinal{\LossM(\wprod(\t_0)) e^{-2r\mu (\t-\t_0)}}
\end{align*}
and thus the claimed estimate.\\

\emph{Proof of \ref{thm:Convergence_iii}:} 
    Since we assume that $ \sup_{t\in [0,T)} \| \wprod(\t) \|_2 < + \infty$, 
    we know by \eqref{eq:with_sigmaout} 
    that\\
    $\sup_{t\in [0,T)} \| \wprod'(\t) \|_2 < + \infty$. 
    \revFinal{Indeed}
    note that \revFinal{by Assumptions \ref{as:A1}, \ref{as:A2}, and \ref{as:A3}}, the right-hand side of \eqref{eq:with_sigmaout} is a continuous function of $\w$.     
    Hence, 
    \begin{equation*}
    \wprod(T) := \lim_{t \to T} \wprod(t) = \lim_{t \to T} \left[ \wprod(0) + \int_{0}^{t} \wprod'(\s) ds \right]
    \end{equation*}
    exists.     
    We thus have one of the following two cases: (i) $\wprod(T) \in \partial \wDomainRP$, i.e., 
    \begin{equation*}
    \lim_{t \to T} \| \w(t) \|_2 = \lim_{t \to T} \| \sigmainn^{-1}(\wprod(t)) \|_2 = \infty
    \end{equation*}
    and the trajectory ends, or (ii) $\wprod(T) \in \wDomainRP$ such that $\lim_{t \to T} \w(t) = \sigmainn^{-1}(\wprod(T))$ and $\w$ can be extended by Peano's theorem and continuity of $\nabla\mathcal{L}$ in \eqref{eq:gd_IRERM}.\\

\emph{Proof of \ref{thm:Convergence_iv}:} Let us now assume that $T=\infty$. We show that if $ \sup_{t\ge 0} \| \wprod(\t) \|_2 < + \infty$,
then $\wprod(\t) \to \wprodinfty$ with $\LossM(\wprodinfty) = 0$. 
The argument is along the lines of \cite{chou2021more} and is repeated here in a simplified form for the reader's convenience.
Let us assume $\wprod$ is bounded, i.e., $\limsup_{\t \to \infty} \| \wprod(\t) \|_2 < + \infty$. By Assumption \ref{as:B2}, there exists a sufficiently large compact ball $B$ around the origin such that $\wprod(\t) \in B$, for all $\t \ge 0$, and $B \cap \{ \widetilde\z \in \R^\N \colon \A\widetilde\z = \y \} \neq \emptyset$.\\
Since $ \left\{ \wprod(\t): t \ge 0 \right\}  \subset B$ and $B$ is compact, there exists a sequence of times $\t_0 < \t_1 < \dots$ such that $\wprod(\t_k) \to \widetilde\z_\star$  as $k \rightarrow \infty$, for some $\widetilde\z_\star \in B$. 
By continuity of $D_F$, we thus get that $\lim_{k\to\infty} D_\F(\widetilde\z_\star,\wprod(\t_k)) = 0$. From \ref{thm:Convergence_i} we furthermore know that $\lim_{\t \to \infty} \A\wprod(\t) = \y$ which implies that $\A\widetilde\z_\star = \y$. Applying Lemma \ref{lem:decay} to $\widetilde\z_\star$ now yields that the function $\t \mapsto D_\F(\widetilde\z_\star,\wprod(\t))$ is non-increasing. Since $D_\F(\widetilde\z_\star,\wprod(\t)) \ge 0$, we obtain that $\lim_{\t\to\infty} D_\F(\widetilde\z_\star,\wprod(\t)) = 0$.\\
Let now $\t_0 < \t_1 < \dots$ be any sequence of times such that $\wprod(\t_k)$ converges to a limit $\widetilde\x$.
Continuity of $D_F$ implies that $D_\F(\widetilde\z_\star,\widetilde\x) = 0$ which yields $\widetilde\x = \widetilde\z_\star$ by strict convexity of $\F$.
Since $\widetilde\z_\star$ is the only accumulation point of $\left\{ \wprod(\t)\right\}_{t \ge 0} \subset B$ and $B$ is compact, we have $\wprod_\infty = \lim_{\t \to \infty} \wprod(\t) = \widetilde\z_\star$ and $\LossM(\wprodinfty) = 0$.\\

\emph{Proof of \ref{thm:Convergence_v}:}
Since $0\le D_F(\widetilde{\z}_0,\wprod_0) < \infty$ by \ref{thm:Convergence_i} and $\t \mapsto D_F(\widetilde{\z}_0,\wprod(\t))$ is non-increasing according to Lemma \ref{lem:decay}, the \revFinal{boundedness of $B_{D_F,r}(\widetilde \z_0)$ implies} that $$\underset{t \in [0,T)}{\sup} \| \wprod(\t) \|_2  < + \infty.$$ As it might be of independent interest, we prove the second claim of \ref{thm:Convergence_v} in the following lemma.
\begin{lemma}
	Let $F(\z) = \langle \bo, H(\z) \rangle$, where $H$ is an antiderivative of $h$ as defined in Theorem \ref{theorem:vector_IRERM}. Assume that Assumption \ref{as:A2} holds and that the univariate function $\widetilde z \mapsto \widetilde z \cdot h'(\widetilde z)$ is in the function class
	\begin{align}
		\label{eq:F_D_lemma}
		\mathcal F_\wDomainRPOneD = \Big\{ f \in L_1^c(\wDomainRPOneD) \colon \Big| \int_{z}^b f(x) dx \Big| = \infty, \text{ for any } z \in \wDomainRPOneD \text{ and } b \in \wDomainRPlim \Big\},
	\end{align}
	where $\wDomainRPlim := \{ \inf \wDomainRPOneD, \sup \wDomainRPOneD \} \cap \{-\infty, \infty\}$. Then, the Bregman ball
	\begin{align*}
		B_{D_F,r}(\x) = \{ \z \in \wDomainRP \colon D_F(\x,\z) \le r \} \subset \R^n
    \end{align*}
    is bounded in the Euclidean metric, for any $\x \in \wDomainRP$ and $r \ge 0$.
\end{lemma}
\begin{proof}
    We assume that $\wDomainRPlim \neq \emptyset$ since otherwise $\wDomainRP$ is bounded and the claim is fulfilled.
	First note that, due to $F(\z) = \langle \bo, H(\z) \rangle$, we have $D_F(\x,\z) = \sum_{i=1}^{N} D_H(x_i,z_i)$. By \ref{as:A2}, $h'(\widetilde z)$ in \eqref{eq:hprimedefinitionWithoutPositivity} is well-defined and strictly positive \revFinal{almost everywhere}, which implies that $H$ is strictly convex.
	Hence, $D_H$ is a nonnegative function and %
	the set $B_{D_F,r}(\x) \subset \R^N$ is unbounded if and only if there exists an index $i\in [N]$ such that the set $B_{D_H,r}(x_i) = \{ z \in \wDomainRPOneD \colon D_H(x_i,z) \le r \} \subset \R$ is unbounded. \\
	For \revFinal{any $p \in \wDomainRPOneD$ and almost every $q \in \wDomainRPOneD$}, we have that 
    $$\tfrac{d}{dq} D_H(p,q) = \tfrac{d}{dq} [ H(p) - H(q) - h(q)(p-q) ] = h'(q) (q-p)$$
    \revFinal{is well-defined. Suppose, for the sake of contradiction, that there exists $p \in \wDomainRPOneD$ such that $B_{D_H,r}(p) = \{ z \in \wDomainRPOneD \colon D_H(p,z) \le r \} \subset \R$ is unbounded. Then, there exists a monotone and diverging sequence $(q_j)_{j\in \mathbb N}  \subset B_{D_H,r}(p)$. Since $B_{D_H,r}(p)$ is convex and thus an unbounded interval, we can replace $(q_j)_{j\in \mathbb N}$ by a (not relabeled) sequence for which $\tfrac{d}{dq} D_H(p,q_j)$ is well-defined, for every $j \in \mathbb N$.}
    
	We first consider the scenario $q_j \to + \infty \in \wDomainRPlim$.
	We have that
	\begin{align}  %
		D_H(p,q_j)
		&= D_H(p,q_0) + \int_{q_0}^{q_j} \left( \frac{d}{d\xi} D_H(p,\xi) \right) d\xi \nonumber \\
		&= D_H(p,q_0) + \int_{q_0}^{q_j} h'(\xi) (\xi - p) d\xi.\label{eq:IntegralIdentity}
	\end{align}
	Since \revFinal{$\frac{h'(\xi_k) (\xi_k - p)}{h'(\xi_k) \xi_k} \to 1$, for any sequence $\xi_k \to \infty$ for which $h'(\xi_k)$ is well-defined, for every $k \in \mathbb N$,} and the terms $h'(\xi) (\xi - p)$ and $h'(\xi) \xi$ are both positive for sufficiently large $\xi$, there exists $\xi_0$ such that $\frac{1}{2} h'(\xi) \xi \le h'(\xi) (\xi - p) \le 2 h'(\xi) \xi$, for \revFinal{almost every} $\xi \ge \xi_0$.
	Note that $\int_{q_0}^{\xi_0} h'(\xi) (\xi - p) d\xi$ is finite as $h' \in L_1^c(\wDomainRPOneD)$ by \ref{as:A2}.
	Since $\widetilde z \mapsto \widetilde z \cdot h'(\widetilde z)$ is in $\mathcal F_\wDomainRPOneD$ and $h' \in L_1^c(\wDomainRPOneD)$ is positive, we know by definition of $\mathcal F_\wDomainRPOneD$ that $\int_{\xi_0}^{\infty} h'(\xi) \xi d\xi$ is equal to $+\infty$. 
    Hence,
	\begin{align*}
		\int_{q_0}^{q_j} h'(\xi) (\xi - p) d\xi
		\ge 
        \int_{q_0}^{\xi_0} h'(\xi) (\xi - p) d\xi
		+ 
        \frac{1}{2} \int_{\xi_0}^{q_j} h'(\xi) \xi d\xi
		\to + \infty,
	\end{align*}
	for $j \to \infty$, which in turn implies that $D_H(p,q_j) \to +\infty$.
    For $q_j \to - \infty \in \wDomainRPlim$ an analogous argument shows that $D_H(p,q_j) \to +\infty$, for $j \to \infty$. \revFinal{Hence, $(q_j)_{j\in \mathbb N}  \not\subset B_{D_H,r}(p)$, a contradiction. We conclude that} $B_{D_H,r}(x_i)$ is bounded, for any $x_i \in \wDomainRPOneD$ and $r > 0$, and the claim follows. 
\end{proof}

\subsection{Proofs of Corollaries \ref{cor:PolynomialRegularization} and \ref{cor:TrigonometricRegularization}}
\label{sec:Non-asymptoticBounds}

In this section we 
present the proofs of Corollaries \ref{cor:PolynomialRegularization} and  \ref{cor:TrigonometricRegularization}.

\begin{proof}[Proof of Corollary \ref{cor:PolynomialRegularization}]
    We will apply Theorem \ref{theorem:vector_IRERM} with $\wDomainRP = \R^\N$. Since we already checked in the proof sketch of Corollary \ref{cor:PolynomialRegularization} that Assumption \ref{as:A2} holds and that $\wprod_{T,\alpha} \in \wDomainRP$, we know that  \eqref{eq:implicit_regularization_Bregman} holds, i.e.,
    \begin{align}
    \label{eq:implicit_regularizationII}
        \wprod_{T,\alpha} \in \argmin_{\widetilde\z\in\wDomainRP, \LossM(\widetilde\z) = 0} \langle \bo, \H(\widetilde\z) - \widetilde\z\odot\h(\wprod_0) \rangle.
    \end{align}
    where $h$ and $H$ have been defined in Theorem \ref{theorem:vector_IRERM} as antiderivatives of $h'(\widetilde z) = \big( [\sigmainn^{-1}]'(\widetilde z) \big)^2$ in \eqref{eq:hPrime}.
    
    Let us calculate the explicit form of \eqref{eq:implicit_regularizationII}. Since the functions $H$ and $h$ are acting entry-wise, we can restrict our derivation to single entries, i.e., $\widetilde\we_0 = \alpha>0$.
    Recalling $p > 1$ we compute that $\h(\widetilde z) = \frac{\p^2}{4(\p-1)} \sign(\widetilde{z}) |\widetilde z|^{\p - 1}$ and $\H(\widetilde z) = \frac{\p}{4(\p-1)} |\widetilde z|^{\p}$. Consequently, 
    \begin{align} \label{eq:polynomialF}
        \H(\widetilde z) - \widetilde z \h(\wprode_0) 
        = \frac{\p}{4(\p-1)}(|\widetilde z|^{\p} - \p\widetilde z \alpha^{\p - 1}).
    \end{align}
    Hence, $\langle \bo, \H(\widetilde\z) - \widetilde\z\odot\h(\wprod_0) \rangle = \frac{p}{4(p-1)} \langle \bo, |\widetilde \z|^{\odot p} - p \alpha^{p-1} \widetilde \z \rangle$ in this case.
    
    Now let $\widetilde\z_\star$ such that $ \Vert \widetilde\z_\star \Vert_p = \min\limits_{\widetilde\z \in\R^\N, \LossM(\widetilde\z) = 0} \|\widetilde \z \|_p $.
    By \eqref{eq:implicit_regularizationII}, we obtain for any $\widetilde\z \in \R^\N$ that
    \begin{equation*}
        \langle \bo, |\wprod_{T,\alpha}|^{\odot p} - p\alpha^{p-1}\wprodinfty  \rangle
        \leq \langle \bo, |\widetilde \z|^{\odot p} - p\alpha^{p-1} \widetilde \z \rangle.
    \end{equation*}
    Rearranging the terms we have for any $\widetilde\z \in \R^N$ with $\| \widetilde\z \|_p \le \|\wprod_{T,\alpha}\|_p$ that
    \begin{align*}
        \|\wprod_{T,\alpha}\|_p^p - \|\widetilde \z\|_p^p
        &\leq p\alpha^{p-1}(\|\wprodinfty\|_1 + \|\widetilde \z\|_1)
        \leq p\alpha^{p-1} N^{1-1/p} (\|\wprod_{T,\alpha}\|_p + \|\widetilde \z\|_p)\\
        &\le 2p\alpha^{p-1} N^{1-1/p} \|\wprod_{T,\alpha}\|_p .
    \end{align*}    
        At the same time, note that it follows from convexity of the univariate function $x \mapsto x^p$ that 
        \begin{align*}
        \|\wprod_{T,\alpha}\|_p^p - \|\widetilde \z\|_p^p
        &\ge
        p \|\widetilde \z\|_p^{p-1} \left(\|\wprod_{T,\alpha}\|_p - \|\widetilde \z\|_p \right).
        \end{align*}
    By combining the last two inequalities we obtain for any $\widetilde\z \in \R^N$ with $\| \widetilde\z \|_p \le \|\wprodinfty\|_p$ that
    \begin{align*}
         \|\wprod_{T,\alpha}\|_p - \|\widetilde \z\|_p 
        \le 
        \frac{2 \alpha^{p-1} }{ \|\widetilde \z\|_p^{p-1}} N^{1-1/p} \|\wprod_{T,\alpha}\|_p.
    \end{align*}
    Rearranging terms yields
    \begin{align*}
       \|\wprod_{T,\alpha}\|_p 
       \le \frac{1}{1- \frac{2 \alpha^{p-1} }{ \|\widetilde \z\|_p^{p-1}} N^{1-1/p}} \|\widetilde \z\|_p 
       \le \left(1+ \frac{4 \alpha^{p-1} }{ \|\widetilde \z\|_p^{p-1}} N^{1-1/p}  \right) \|\widetilde \z\|_p, 
    \end{align*}
    where in the second inequality we used the elementary estimate $ \frac{1}{1-x} \le 1+2x $ for $0<x \le 1/2$ and our assumption on $\alpha$.
    By setting $\widetilde\z=\widetilde\z_\star$ the claim follows.
    
\end{proof}

Finally, let us present the formal proof of Corollary \ref{cor:TrigonometricRegularization}.

\begin{proof}[Proof of Corollary \ref{cor:TrigonometricRegularization}]
    As in the proof of Corollary \ref{cor:PolynomialRegularization}, we will apply Theorem \ref{theorem:vector_IRERM}. Again, we only need to verify Assumption \ref{as:A2}. 
    If in addition $\wprodT \in \wDomainRP$, which trivially holds for $\sigmainn = \sinh$ but needs to additionally be verified for $\sigmainn = \tanh$, Theorem \ref{theorem:vector_IRERM} may be applied and yields that \eqref{eq:implicit_regularization_Bregman} holds, i.e.,
    \begin{align}
    \label{eq:implicit_regularizationIII}
        \wprodT \in \argmin_{\widetilde\z\in\overline{\wDomainRP}, \LossM(\widetilde\z) = 0} \langle \bo, \H(\widetilde\z) - \widetilde\z\odot\h(\wprod_0) \rangle, 
    \end{align}
    where $h$ and $H$ have been defined in Theorem \ref{theorem:vector_IRERM} as antiderivatives of
    \begin{align*}
        h'(\widetilde z) = \big( [\sigmainn^{-1}]'(\widetilde z) \big)^2.
    \end{align*}

    \emph{Proof of \ref{cor:TrigonometricRegularizationNonAsymptotic_I}:} We first consider $\sigmainn(\z) = \sinh(\z)$ with $\wDomainRP = \R^\N$. In this case, $\wprodT \in \wDomainRP$ is trivially satisfied by existence of the limit. Since the function $\sigmainn$ is defined entry-wise, both $H$ and $h$ will be defined entry-wise and we can restrict our derivation to single entries.
    For $\widetilde z = \sigmainn(\ze) = \sinh(\ze)$, %
    we have that
    \begin{align}
    \label{eq:h_prime_sinh}        
        \h'(\widetilde z) 
        = \big( [\sigmainn^{-1}]'(\widetilde z) \big)^2 
        = \frac{1}{ \sigmainn'( z)^{2} }
        = \frac{1}{\cosh(z)^{2}} = \frac{1}{1+\sinh(z)^{2}} = \frac{1}{1+\widetilde z^{2}}.        
    \end{align}
    Since $\sigmainn$ is continuously differentiable and invertible, the derivative $\sigmainn'$ never vanishes, and $h' \in L_1^c(\R^\N)$, Assumption \ref{as:A2} holds. 
    As $\sigmainn' (x) \neq 0$ for all $ x \in \R$, any solution of \eqref{eq:gd_IRERM} is regular and Theorem \ref{theorem:vector_IRERM}  can be applied 
    to obtain \eqref{eq:implicit_regularizationIII} with $\h(\widetilde z) = \arctan(\widetilde z)$, $\H(\widetilde z) = \widetilde z\arctan(\widetilde z) - \frac{1}{2}\log(1+\widetilde z^{2})$, and
    \begin{align*}
        \H(\widetilde\ze) - \widetilde\ze\h(\wprode_0)
        = \widetilde\ze(\arctan(\widetilde\ze) - \arctan (\wprode_0)) - \frac{1}{2}\log(1+\widetilde\ze^{2}).
    \end{align*}
    Observing that $\langle \bo, \H(\widetilde\z) - \widetilde\z\odot\h(\wprod_0) \rangle = \g_{\sinh}(\widetilde\z)$, for $\wprod_0 = \0$, concludes the proof of \ref{cor:TrigonometricRegularizationNonAsymptotic_I}.\\ 

    \emph{Proof of \ref{cor:TrigonometricRegularizationNonAsymptotic_II}:} We now consider $\sigmainn(\z) = \tanh(\z)$ with $\wDomainRP = (-1,1)^\N$, and assume in addition that $\sigmaout = \mathrm{Id}$.
    First note that, for $\widetilde z = \sigmainn(\ze) = \tanh(\ze)$,
    \begin{align*}        
        \h'(\widetilde z) = 
        \big( [\sigmainn^{-1}]'(\widetilde z) \big)^2 
        =\frac{1}{\sigmainn'(\ze)^{2}} = \frac{1}{(1-\tanh(\ze)^2)^2}
        = \frac{1}{(1-\widetilde z^2)^2}.        
    \end{align*} 
    By partial fraction decomposition, we obtain that
    \begin{align*}
        \h(\widetilde z)&= \frac{1}{4}\left(\log(1+\widetilde z) - \log(1-\widetilde z) + \frac{2\widetilde z}{1-\widetilde z^2}\right) \\
        &= \frac{1}{2}\left(\artanh(\widetilde z)  +\frac{\widetilde z}{1-\widetilde z^2}\right)
    \end{align*}
    and 
    \begin{align*}
        \H(\widetilde z)
        = \frac{1}{2} \widetilde z \ \artanh (\widetilde z).
    \end{align*}
    Consequently, 
    \begin{align}
    \label{eq:RandomEquation}
        \H(\widetilde\ze) - \widetilde\ze\h(\wprode_0)
        = \frac{\widetilde\ze}{2}\left(\artanh (\widetilde\ze) -\artanh (\wprode_0) - \frac{\wprode_0}{1-\wprode_0^2}\right),
    \end{align}
    which shows that $\langle \bo, \H(\widetilde\z) - \widetilde\z\odot\h(\wprod_0) \rangle = \frac{1}{2}  \g_{\tanh}(\widetilde\z)$ for $\wprod_0 = \0$. \\
    To apply Theorem \ref{theorem:vector_IRERM}, we have to verify that \ref{as:A2} holds, that $\w$ is regular, and that $\wprodT \in \wDomainRP$. Since $\sigmainn$ is continuously differentiable, invertible and the derivative $\sigmainn'$ never vanishes, and since $h' \in L_1^c([-1,1]^\N)$, we have that Assumption \ref{as:A2} holds.
    As $\sigmainn' (x) \neq 0$ for all $ x \in \R$, any solution of \eqref{eq:gd_IRERM} is regular. Finally, from Corollary \ref{cor:TrigonometricRegularizationConvergence} we know that $\wprodT \in \wDomainRP$ and Theorem \ref{theorem:vector_IRERM}  can be applied. This concludes the proof.
\end{proof}

\section*{Acknowledgements}

The authors want to thank Axel B\"ohm for helpful comments. H.C. acknowledges the funding from the Munich Center for Machine Learning (MCML). J.M. acknowledges partial support from the Munich Center for Machine Learning (MCML) and the Deutsche Forschungsgemeinschaft (DFG, German Research Foundation) – GRK 3081/1 – Project number 534429653.

\section*{Statements and Declarations}

The authors declare that they have no conflict of interest.

\bibliography{references}{}
\bibliographystyle{abbrv}

\end{document}